\documentclass[letterpaper, 10 pt, conference]{ieeeconf}

\IEEEoverridecommandlockouts
\usepackage{cite}
\usepackage{amsmath,amssymb,amsfonts}
\usepackage{algorithmic}
\usepackage{graphicx}
\usepackage{textcomp}
\usepackage{xcolor}
\usepackage[makeroom]{cancel}

\let\labelindent\relax

\usepackage{enumerate}
\usepackage{enumitem}

\usepackage{amsthm}
\usepackage{amsmath}
\usepackage{thmtools}
\usepackage{semantic}
\usepackage{cleveref}
\usepackage{amsfonts}
\usepackage[scr=rsfs]{mathalpha}
\usepackage[thinc]{esdiff}

\declaretheorem[]{lemma}
\declaretheorem[name = Theorem]{thm}
\declaretheorem[name = Proposition]{proposition}
\declaretheorem[name = Assumption]{assump}
\declaretheorem[name = Remark]{rmk}

\newtheorem{algdes}{Setup}

\newcommand{\R}{\mathbb{R}}

\newcommand{\N}{\mathbb{N}}
\newcommand{\w}{\omega}
\newcommand{\D}{\mathscr{D}}
\newcommand{\E}{\mathbb{E}}

\makeatletter
\newcommand{\distas}[1]{\mathbin{\overset{#1}{\kern\z@\sim}}}%
\newsavebox{\mybox}\newsavebox{\mysim}

\def\BibTeX{{\rm B\kern-.05em{\sc i\kern-.025em b}\kern-.08em
    T\kern-.1667em\lower.7ex\hbox{E}\kern-.125emX}}
\begin{document}

\title{Time-delay Induced Stochastic Optimization and Extremum Seeking \\
}

\author{Naum Dimitrieski$^{1}$, Michael Reyer$^{1}$, Mohamed-Ali Belabbas$^{2}$, and Christian Ebenbauer$^{1}$
\thanks{Naum Dimitrieski, Michael Reyer, and Christian Ebenbauer are with the Chair of Intelligent Control Systems, RWTH Aachen University, 52062 Aachen, Germany
        {\tt\small (e-mail: \{naum.dimitrieski, michael.reyer, christian.ebenbauer\}@ic.rwth-aachen.de)}}%
\thanks{Mohamed-Ali Belabbas is with the Coordinated Science Laboratory, University of Illinois at Urbana-Champaign, Urbana, IL 61801 USA 
        {\tt\small (e-mail: belabbas@illinois.edu)}}%
}

\maketitle

\begin{abstract}

In this paper a novel stochastic optimization and extremum seeking algorithm is presented, one which is based on time-delayed random perturbations and step size adaptation. 
For the case of a one-dimensional quadratic unconstrained optimization problem, global exponential convergence in expectation and global exponential practical convergence of the variance of the trajectories are proven. 
The theoretical results are complemented by numerical simulations for one- and multi-dimensional quadratic and non-quadratic objective functions. 
\end{abstract}

\section{Introduction}

Extremum seeking (ES) is an adaptive control concept that allows the control of unknown dynamical systems. In its simplest form, ES algorithms are a class of gradient-free iterative optimization algorithms for unconstrained optimization problems that only require the values (and not the gradient) of the objective function for optimization. Over the past two decades, extremum seeking control has undergone a significant development, and today, there exist several algorithms and methodologies to systematically analyze and 
synthesize extremum seeking controllers~\cite{nesic2010survey, scheinker2024survey}.

Our motivation for this work is to develop a novel stochastic and discrete-time ES algorithm
that is partly inspired by Lie bracket-based deterministic ES algorithms.
In more detail, the simplest Lie bracket-based deterministic continuous-time ES algorithm is of the form~\cite{durr2013lie}, \cite{labar2018gradient}
    \begin{align}
    \label{ctes}
        \dot{x}(t) = -\sqrt{\w}\sin(\w t) J(x(t)) + \sqrt{\w}\cos(\w t)
    \end{align}
    with the objective function $J: \R \to \R$.
    If the objective is a continuously differentiable function, then it can be shown
    that the trajectories of \eqref{ctes}, for a sufficiently large frequency~${\omega>0}$, approximate (uniformly in the $L_\infty$-norm) trajectories  of the gradient dynamics $\dot{\bar x}(t)=-\frac{1}{2}\nabla J(\bar{x}(t))$.
    The terms $\sin(\w t),\cos(\w t)$ are so-called dither (learning, exploration) signals
    which enable the non-parametric learning of the gradient by averaging.
    
    Considering the sine term as a delayed cosine term (in the sense of 
    $\cos(t-\pi/2)=\sin(t)$), the first main idea of this work is to replace the cosine term by zero-mean noise and the sine term by a delayed 
    version of the zero-mean noise. This idea leads in the continuous-time setting to
    a nonlinear stochastic delay-differential equation of the form 
    \begin{align}
        \mathrm{d}x(t)=-J(x(t)) \mathrm{d}w(t-T)+ \mathrm{d}w(t),
    \end{align}
    where~$\mathrm{d}w$ is a standard Wiener process and $T > 0$ is the time delay. While this continuous-time equation has many interesting properties and applications,    as reported in \cite{ali_christian_report}, we focus in this paper on discrete-time ES algorithms. Hence, with zero-mean real random variables~$w_k, k \geq 0$, a discrete-time algorithm could take the form
    \begin{align}
    \label{dtses}
        x_{k+1} - x_k = - w_{k-1}J(x_k) + w_k,
    \end{align}
    where $w_i,w_j$, $i \neq j$, are stochastically independent random variables. 
    While the continuous-time system  \eqref{ctes} is shown to be semi-globally asymptotically practically stable with respect to the global minimizer of a  convex objective function \cite{durr2013lie},  the discrete-time stochastic system \eqref{dtses} is not suitable for
    optimization due to an unbounded growth of the variance over time.

    The main contribution of this paper is the introduction of a novel discrete-time stochastic optimization and ES algorithm which possesses favorable
    properties in terms of convergence. This is achieved by an appropriate modification of \eqref{dtses}. In particular, we present a class of
    discrete-time stochastic ES algorithms which achieve global exponential convergence in expectation and global exponential practical convergence of the variance for a quadratic objective function. 

    Due to the use of delayed zero-mean random exploration, the proposed ES algorithm does not require a tuning of the exploration sequence, as it is for example necessary for the discrete-time deterministic ES algorithm presented in \cite{feiling_ali_christian2021gradient}. In addition, it only requires objective function evaluations at~$x_k$, i.e., $J(x_k)$,
    and not at e.g. $x_k+\delta_k$, where $\delta_k$
    are either random perturbations or user-defined dither functions, as for instance \cite{spall_book2005introduction,manzie_krstic2009extremum, liu_krstic2015stochastic, nesic2014multi, hazeleger_nesic2022sampled}, in order to perform an iteration step in the optimization.
    In more detail, in \cite{spall_book2005introduction,manzie_krstic2009extremum, liu_krstic2015stochastic} the objective function is perturbed using a zero-mean random variable, in \cite{nesic2014multi} for the gradient-based method canonical basis vectors are used for objective function perturbation, and in \cite{hazeleger_nesic2022sampled} a broader class of (possibly) input-dependent dither signals is considered.
    Thus, the proposed algorithm is in particular appealing for ES problems, since in ES problems~$x_k$ typically represents the state of a dynamical
    system and evaluations at $x_k+\delta_k$ would require additional effort to steer the system from state $x_k$ to $x_k+\delta_k$. 
    
    The second main idea and important novelty of the proposed algorithm is
    an adaptive step size rule (which is state-dependent and involves an additional dynamic equation). Hence, the step size sequence is not a priori determined to be a decreasing one, as it is the case in a typical implementation of a Finite Difference Stochastic Approximation (FDSA) method \cite{kiefer_wolfowitz1952stochastic}, see e.g. the Simultaneous Perturbation Stochastic Approximation (SPSA) sequence in \cite{spall1992multivariate}. This adaptation is crucial to control the variance and to achieve (practical) convergence in the mean-square sense.
    Moreover, the adaptation allows for the possibility of optimizing a time-varying problem, as this is common in ES control and online learning.
    Finally, time-delays in the context of ES also have been studied, for example in \cite{zhu_fridman2022extremum, oliveira_krstic2016extremum}, in order to carry out an averaging analysis or to deal with systems with delays, but not as a mean to synthesize algorithms like in this work.
    
    In terms of application, apart from ES problems\cite{scheinker2024survey}, there exists the possibility of utilizing the proposed algorithm class in optimization problems where no gradient information is available. For example, it could represent an alternative approach to policy gradient-like algorithms \cite{spall_book2005introduction,meyn2022control}, but this is not subject of the presented work. 

    The structure of the paper is as follows. The problem statement and main result of this paper are given in \Cref{section_II}, and in \Cref{section_V} we provide some numerical results for multiple optimization problems. In \Cref{section_VI} we provide the conclusions of this paper and the outlook on potential future research directions. The proof of the main convergence result as well as further auxiliary results are presented in the \textcolor{black}{Appendix}.

    \textit{Notation:} The notation is rather standard. We denote the sets of real and positive numbers as $\R$ and~$\R^{+}$. Moreover, the expectation of any random variable  is denoted by $\E[\cdot]$, and the abbreviation i.i.d. is used for identical and independently distributed random variables. 
    $P \succ 0$ denotes a symmetric and positive definite matrix $P$
    and $\mathcal{C}^{2}$ denotes the class of twice continuously differentiable
    functions.

\section{Problem Statement and Main Result}
\label{section_II}
In this paper, we consider the problem of finding a global minimizer $x^{*} \in \R$ of a $\mathcal{C}^{2}$ function $J: \R \to \R$, i.e.,
\begin{equation}
\label{problem}
    x^{*} \in \underset{x \in \R}{\mathrm{argmin}}~J(x).
\end{equation}
To solve \eqref{problem}, we propose a discrete-time algorithm of the form
\begin{subequations}
    \label{sys_def}
    \begin{align}
        \label{sys_def_x}
        x_{k+1} &= x_k - \rho y_k + \hat{g}(y_k,w_{k}) \\
        \label{sys_def_y}
        y_{k+1} &= (1-\beta)y_k \notag  \\
        &+ \hat{h}(y_{k-1},w_{k-1})\big(J(x_k) - J(x_{k-1})\big), 
    \end{align}
\end{subequations}
where $\hat{g}, \hat{h}: \R^2 \to \R$, $\rho \in \R^{+}$, $\beta \in (0,2)$ and~$w_k \in \R$ is an i.i.d. zero-mean random variable.  We refer to $k \in \mathbb{N}$ 
as time steps and we refer to \eqref{sys_def_x} and \eqref{sys_def_y} as system \eqref{sys_def}. We will assume throughout this paper that $J$ has a {\em unique} global minimizer, and for the sake of presentation, we will assume it is moreover quadratic. 
In the following, we provide further
design elements for \eqref{sys_def} (presented together as a "Setup"), and an assumption on the objective function
which will ease our convergence analysis. We also provide a more in-depth discussion regarding the motivation for them in this section.
\begin{algdes}
    \label{assump:ghw}
    In system \eqref{sys_def},
    \begin{enumerate}[label=(\alph*)]
        \item \label{assump:ghw:1} for some $\varepsilon \in \R^{+}$, the functions ${\hat{g},\hat{h}: \R^2 \to \R}$ are defined as $\hat{h}(y,w) := \frac{h(w)}{|y| + \varepsilon}$ and $\hat{g}(y,w) := (|y| + \varepsilon)g(w)$, where $h,g: \R \to \R$ are odd functions.
        Furthermore, for all $w \in \R$ it holds that $\mathrm{sign}(g(w))=\mathrm{sign}(h(w))$ and $g(w)=h(w) = 0$ if and only if $w = 0$.

        \item \label{assump:ghw:2} $w_i \in \D:= \{-\w,\w\}$, $\w >0$, 
        $i \in \mathbb{N} \cup \{0\}$, is an i.i.d. discrete
        random variable with equal probability taking 
        the value~$-\w$ or $\w$.
    \end{enumerate}
\end{algdes}
\begin{assump}
    \label{assump:quadratic}
    The objective function $J$ in system \eqref{sys_def} is quadratic with~$J(x) := J^{*} + \frac{\mu}{2}(x-x^{*})^2$, where $\mu \in \R^{+}$, and~${x^{*},J^{*} \in \R}$. 
\end{assump}
We consider a quadratic objective function because 
many objective functions can be locally (around a local minimum) well approximated
by a quadratic function
and, consequently, it can be shown that our algorithms in fact find local minima of  
$\mathcal{C}^2$ functions given proper parameter tuning and suitable assumptions. However, due to space constraints and for clarity of exposition, we only deal with the case of a quadratic objective function in this paper.

For system \eqref{sys_def} with \Cref{assump:ghw}, under the above assumption, we aim to show that the system iterates $(x_k,y_k)$ converge in expectation to $(x^{*}, 0)$, with their variance converging to an~$\varepsilon$-neighbourhood of $0$.

The system \eqref{sys_def} is motivated by the need to control the variance, as for system \eqref{dtses} this is probably not possible. By adjoining an additional equation, where $y_k$ may be thought of as a gradient-storing moment term, it becomes easier to control the variance and to ensure it converges practically, as it will become clear in the proof of the main result. The difference between the behavior of these two systems is also shown in simulations in \Cref{section_V}, where, through comparison, the importance of the $y_k$-equation and the adaptive step size becomes evident. Similarly, it has been shown that by adding a low-pass filter the performance of ES systems may be improved (see for example \cite{labar2018gradient}). 
Furthermore, notice that the dependence of $\hat h$ on the delayed random variable $w_{k-1}$ is vital. If instead $\hat h$ would depend on $w_k$,
then it is easy to see, for example by considering the expectation of 
\eqref{sys_def}, that the system does not show an optimizing behavior.

\Cref{assump:ghw} \ref{assump:ghw:1} can be associated to phase shifted periodic exploration functions used for ES.
A simple choice for $h,g$ is the identity map, i.e., $h(w)=g(w)=w$.
Division and multiplication by $|y|+\varepsilon$ in $\hat h,\hat g$ is needed to control the variance, and the small positive constant $\varepsilon > 0$ is introduced to avoid division by zero. 
(These constructions become clear when studying the proof of our main result.) Moreover, \Cref{assump:ghw} \ref{assump:ghw:2} is motivated by replacing deterministic exploration sequences, as the ones for example in \cite{feiling_ali_christian2021gradient}, by stochastic ones. 
Further, although restrictive, \Cref{assump:quadratic} is used to keep the convergence proof rather simple. 

Throughout the paper, for the sake of clarity, we use the following abbreviations {${h_k := h(w_k)}$}, {${g_k:= g(w_k)}$}, with~{${\chi := \E\left[h_k^2\right]}$}, and~{$\psi := \E\left[g_k^2\right]$}, as well as~${\tilde{x}_k := x_k - x^{*}}$ at any $k \in \N \cup \{0\}$.

The main result of this paper is the exponential convergence in expectation of the system trajectories for any initial conditions $x_0, y_0 \in \R$, and exponential practical convergence of the variance of the system trajectories. 

\begin{thm}
    \label{theorem:convergence_means_square}
    Consider system \eqref{sys_def} with \Cref{assump:ghw}, suppose \Cref{assump:quadratic} holds, and let $\chi, \psi, \mu \in \R^{+}$ be arbitrary.
    Assume that~${\rho, \beta,\varepsilon}$ satisfy
    \begin{enumerate}
        \item $\beta > \mu\rho\sqrt{\chi\psi}$,
        \item $\rho \in (0, \rho_0]$, $\beta \in [1-\beta_0, 1+\beta_0]$, and $\varepsilon \in (0, \varepsilon_0]$,
    \end{enumerate}
    where ${ \rho_0, \beta_0, \varepsilon_0  \in \R^{+}}$ are sufficiently small. Then, for any~$x_0,y_0 \in \R$ the system trajectory~${(x_k, y_k)}$ converges exponentially in expectation to~$(x^{*},0)$ and its variance converges exponentially to an {$\varepsilon$-neighbourhood} of~$0$.
\end{thm}

The proof of \Cref{theorem:convergence_means_square} is rather lengthy and can be found in the \textcolor{black}{Appendix}, which contains several additional technical results about the proposed system \eqref{sys_def}. For example, the equations for the dynamics of the expectation as well as for the dynamics of the variance of \eqref{sys_def} are presented and derived in the \textcolor{black}{Appendix}, which are central for the convergence and stability properties established in \Cref{theorem:convergence_means_square}.
In addition, a result on the convergence on sequences upper- and lower-bounded by perturbed stable linear time-invariant systems is established (see Proposition \ref{lemma:matrix_upper_lower_bounds}).
Due to space limitation, we present these results and proofs in the \textcolor{black}{Appendix}.

\section{Simulations}
\label{section_V}
In this section, we present  numerical results we have obtained by implementing systems \eqref{dtses} and \eqref{sys_def}. The former is also considered to highlight the differences in performance with respect the latter. For simulations, we used~${2\cdot 10^{5}}$ sample trajectories for the one-dimensional optimization problems and\textcolor{black}{, due to memory restrictions,}  $10^{5}$ for the multidimensional one to estimate the expectation and variance of the sample trajectories. Furthermore, the term {$1$-$\sigma$} denotes the square root of the variance of the sample trajectories, i.e. one standard deviation. Moreover, in all simulations, we set~${\varepsilon  = 10^{-7}}$, and
\begin{align*}
    \E\left[w_k^2\right] &:= 1, & g(w_k) &:= \sqrt{\psi}w_k, & h(w_k) &:= \sqrt{\chi}w_k, 
\end{align*}
where $\chi$ and $\psi$ are positive parameters adapted for each simulation. Next, for each simulation we use one of the following objective functions
\begin{subequations}
    \begin{align}
        \label{one_dim_q}
        J(x) &= 10000 + (x-x^{*})^2, \\
        \label{x2cosx}
        J(x) &= 10000 + x^2\cos(0.2x), \\
        \label{one_dim_convex}
        J(x) &= 10000 + \log(1+e^{x-x^{*}}) + \log(1+e^{x^{*}-x}),\\
        \label{multi_dim_q}
        J(x) &= 10000 + (x - x^{*})^{\top}H(x-x^{*}),
    \end{align}
\end{subequations}
where $H \succ 0$, and $x^{*} \in \R$ is chosen differently for each simulation. In addition, due to space constraints, we show predominantly the numerical results for the state $x$ (and only once for $y$).

We first present the numerical results for the objective function \eqref{one_dim_q}, with $x^{*} = 25$ and initial conditions $x_0 = -40$, and $y_0 \distas{} \mathcal{U}(-5,10)$, with $y_0$ being sampled only once prior to the execution of all the simulations. Setting 
\begin{align}
\begin{split}
    \label{param:1}
    \rho = 0.12,~~\chi = \frac{121}{4}, ~~ \psi = 0.01, ~~\beta = 0.75, 
\end{split}    
\end{align}
we obtain the results shown on \Cref{fig:quadratic_1}. 
We observe the $x$-trajectory in expectation to converge to $x^{*}$, with the variance decreasing. To compare the theoretical results (obtained in the \textcolor{black}{Appendix}) with the empirically obtained ones, for the given initial conditions we first simulate the dynamics of the expectation and the upper bound dynamics of the variance given by
\begin{align*}
    \E\left[\begin{matrix}
        \Tilde{x}_{k+1} \\
        y_{k+1}
    \end{matrix}\right] &= \begin{bmatrix}
        1 & -\rho \\
        \mu\gamma & 1-\beta
    \end{bmatrix}\E\left[\begin{matrix}
        \Tilde{x}_{k} \\
        y_{k}
    \end{matrix}\right], \\
    \E\left[\zeta_{k+1}\right] &= (A_{ms} + \varepsilon Q_2)\E\left[\zeta_k\right] + \varepsilon b_2,
\end{align*}
with $\zeta_k$ and $A_{ms}$ as defined in the \textcolor{black}{Appendix}, and~${Q_2,b_2}$ as defined in the proof of \Cref{theorem:convergence_means_square} (see \textcolor{black}{Appendix}). Here, for the simulations we used $\varepsilon_0 = \varepsilon$. Thereafter, we calculate the~{$1$-$\sigma$} bound at every~$k$ as~${\sqrt{\E\left[\Tilde{x}_k^2\right] - \left(\E\left[\Tilde{x}_k\right]\right)^2}}$. These results, alongside the empirically obtained $1$-$\sigma$ bound are plotted on \Cref{fig:quadratic_1.2}, where we observe the empirically obtained $1$-$\sigma$ bound to be almost identical to the theoretically obtained upper bound on $1$-$\sigma$, albeit still lower.
\begin{figure}[!t]
    \centering
    \includegraphics[width=8cm]{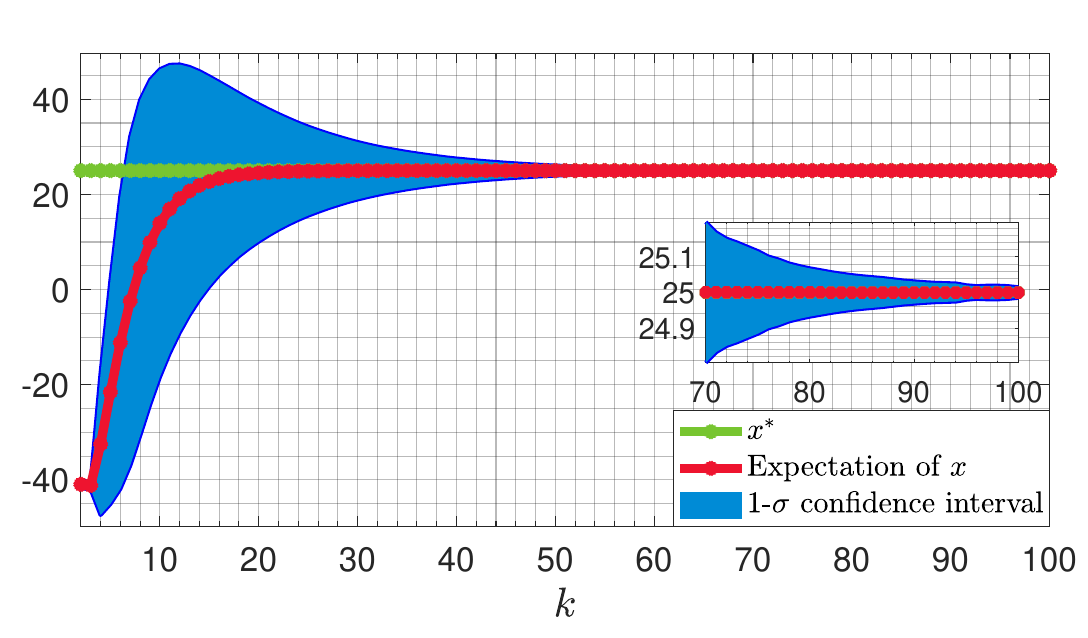}
    \caption{Simulation result for system \eqref{sys_def} with objective function  \eqref{one_dim_q}
    }
    \label{fig:quadratic_1}
\end{figure}
\begin{figure}[!t]
    \centering
    \includegraphics[width=8cm]{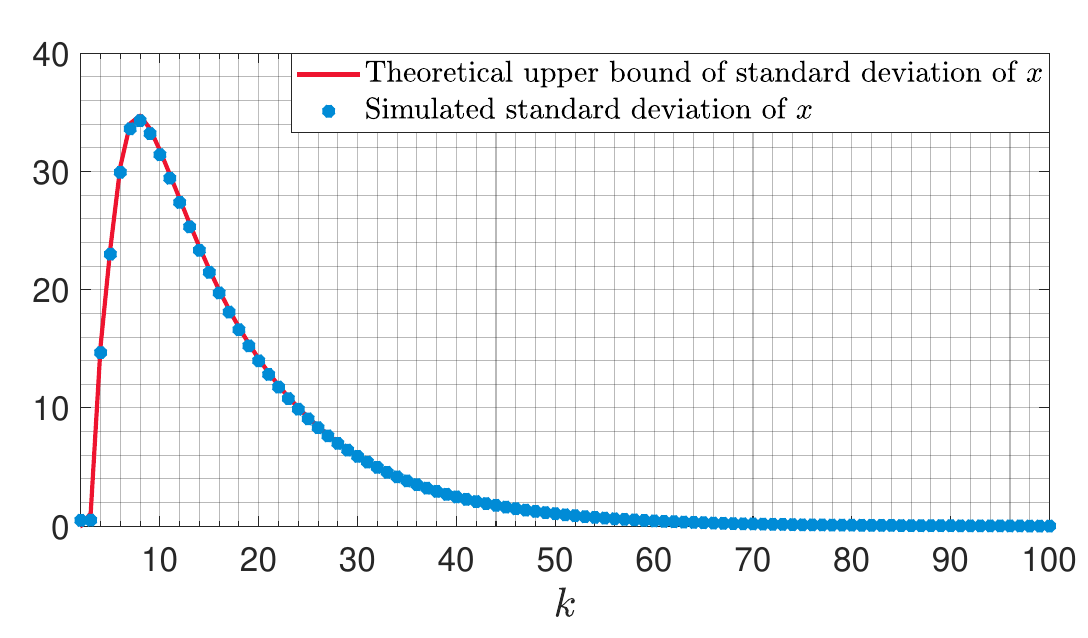}
    \caption{Standard deviation of system \eqref{sys_def}
    with the objective function \eqref{one_dim_q} }
    \label{fig:quadratic_1.2}
\end{figure}
To further illustrate the obtained results, we present \Cref{fig:quadratic_1.3} and \Cref{fig:quadratic_1.4}. There we observe $5$ randomly sampled trajectories generated by \eqref{sys_def} from the simulation setup of \Cref{fig:quadratic_1}, where each trajectory $(x_k,y_k)$ is plotted with one corresponding colour on both figures. We observe that the $x$ trajectories converge to $x^{*} = 25$, and the $y$ trajectories converge to $y_E = 0$.
\begin{figure}[!t]
    \centering
    \includegraphics[width=8cm]{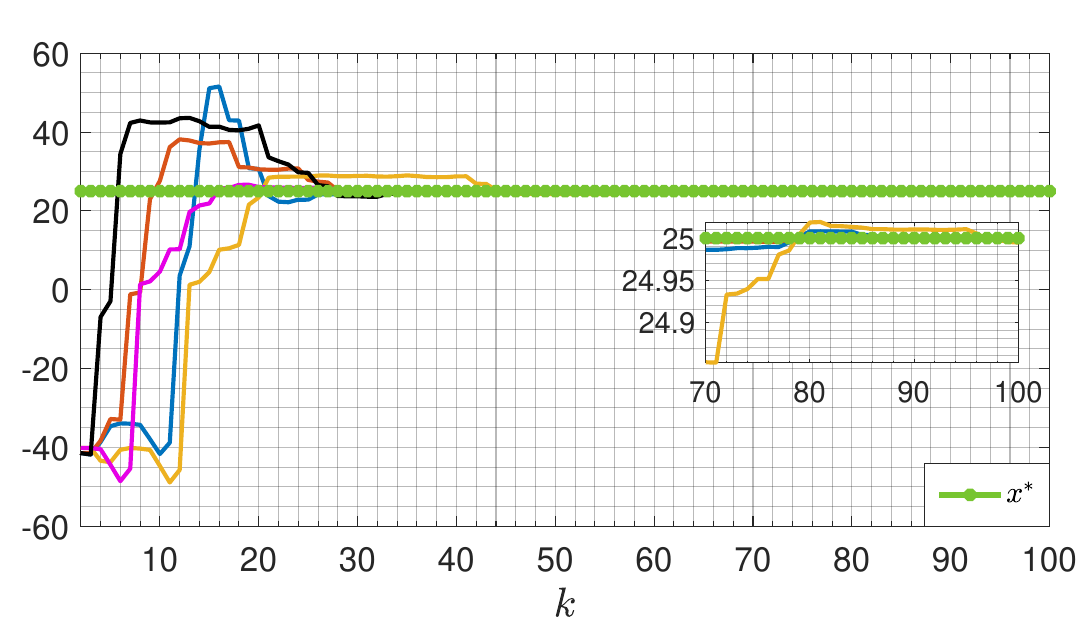}
    \caption{Sample trajectories of the $x-$subsystem in \eqref{sys_def} with objective function  \eqref{one_dim_q}
    }
    \label{fig:quadratic_1.3}
\end{figure}
\begin{figure}[!t]
    \centering
    \includegraphics[width=8cm]{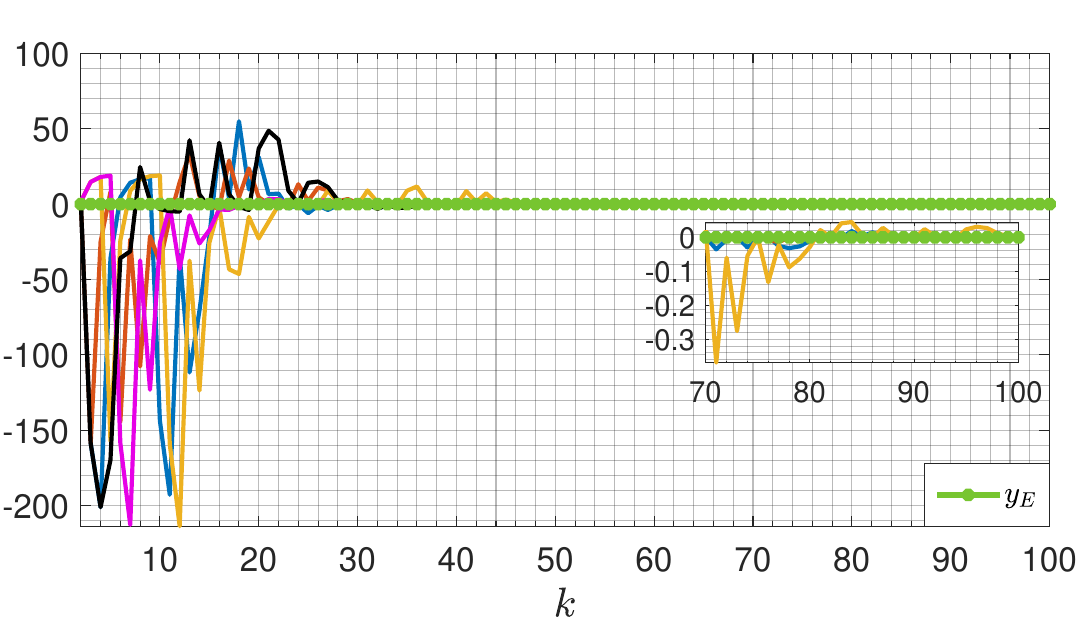}
    \caption{Sample trajectories of the $y-$subsystem in \eqref{sys_def} with objective function  \eqref{one_dim_q}
    }
    \label{fig:quadratic_1.4}
\end{figure}

Next, to further emphasize that the convergence result is global, we provide another numerical result, shown on \Cref{fig:quadratic_2}, where we use the same parameters as in \eqref{param:1}. However, for the objective function \eqref{one_dim_q} here we set~${x_0 = -4\cdot10^5}$ and $x^{*} = 2.5\cdot10^5$, with $y_0$ initialized as in the previous simulation. We here note the high level of similarity between the plots in \Cref{fig:quadratic_1} and \Cref{fig:quadratic_2}, which supports the claim that the trajectory convergence type is indeed global and exponential in expectation and variance. 
\begin{figure}[!t]
    \centering
    \includegraphics[width=8cm]{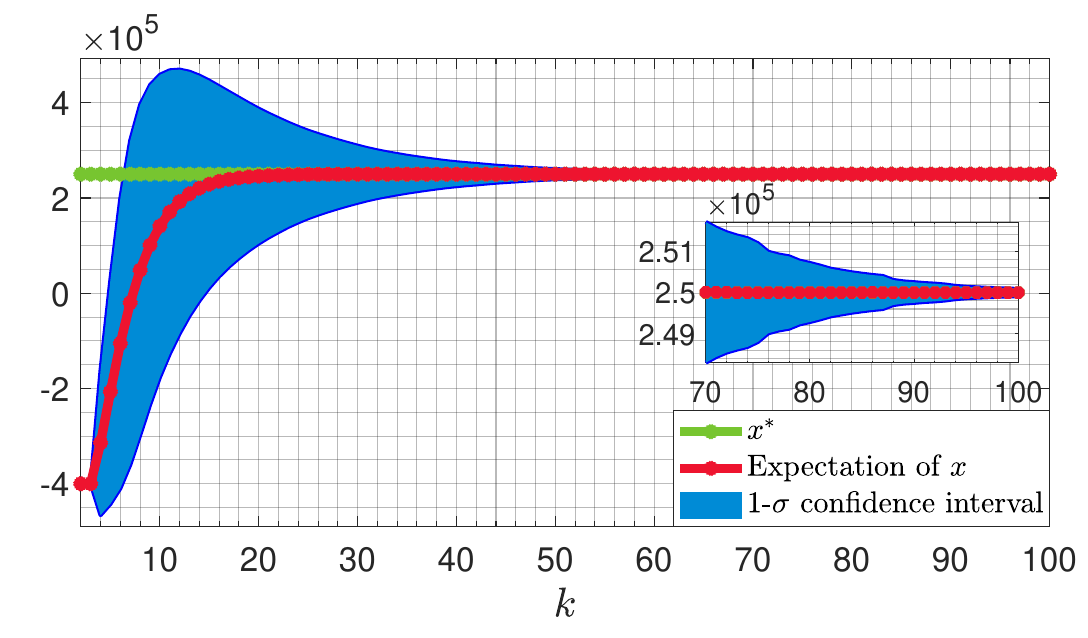}
    \caption{Simulation result for system \eqref{sys_def} with objective function \eqref{one_dim_q}}
    \label{fig:quadratic_2}
\end{figure}
To provide numerical evidence that \Cref{assump:quadratic} can be relaxed, we present in \Cref{fig:x2cosx} simulation results for the non-quadratic objective \eqref{x2cosx} whose Taylor expansion around $x_1^{*} = 0$ and~${x_2^{*} = 48.15}$ approximates a quadratic function. For the local optimization problem in the neighbourhood of $x_1^{*}$ we use the same simulation parameters as in \eqref{param:1}, and for the one in the neighbourhood of $x_2^{*}$ we set $\rho = 0.05$ and~${\chi = 0.09}$, with the other parameters remaining the same. In addition, we chose the initial conditions as~${x_0 = -2}$ and~${x_0 = 40}$ respectively, and with $y_0$ chosen for both cases as in the previous simulations.
One can observe that the sample trajectories show great similarity (in terms of mean and $1$-$\sigma$ bound) to the ones in \Cref{fig:quadratic_1}, and moreover the algorithm converges to the local minimizer.
\begin{figure}[!t]
    \centering
    \includegraphics[width=8cm]{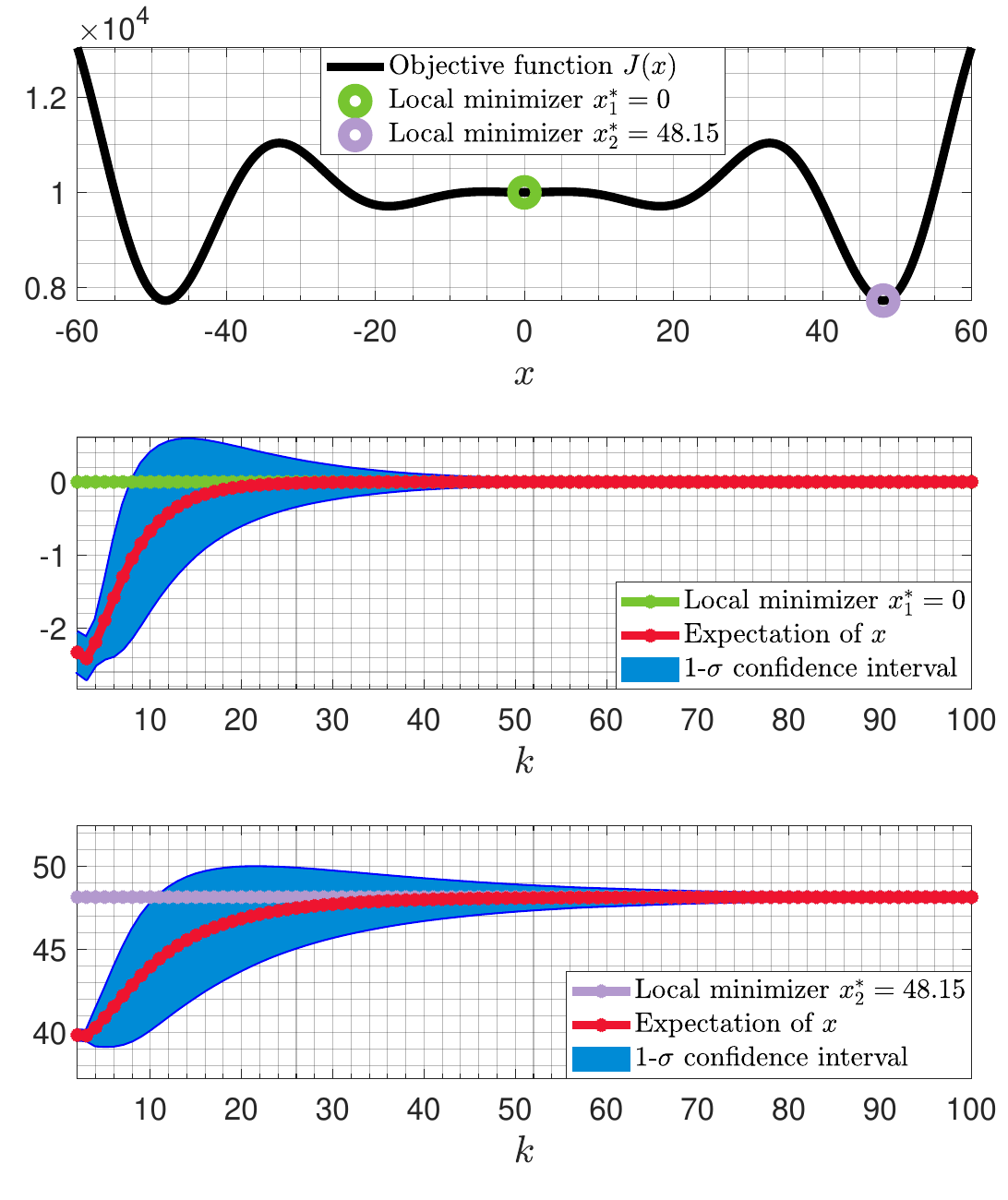}
    \caption{Simulation result for system \eqref{sys_def} with objective function \eqref{x2cosx}, for two different initial conditions}
    \label{fig:x2cosx}
\end{figure}

Next we provide numerical evidence that the global exponential practical type of convergence result is due to the use of the adaptive step size, i.e. the maps $\hat{h}$ and $\hat{g}$. Hence, we present \Cref{fig:quadratic_2.2}, where we do not use an adaptive step size, rather we set 
\begin{align*}
    \hat{h}(y,w)  &= h(w) = \sqrt{\chi}w, & \hat{g}(y,w)  &= g(w) = \sqrt{\psi}w,
\end{align*}
with $\rho = 0.05$, $\chi = 0.81$, and $ \beta = 0.4$, 
and where the remaining parameters are as in the previous simulations. We observe on \Cref{fig:quadratic_2.2} that the trajectories differ from the previously shown ones in the sense of the existence of a steady state variance, and moreover convergence is not globally ensured, rather for a different pair of $(x_0, x^{*})$ and a different choice of $y_0$ divergence may occur. This is expected, as our analysis (which is not presented in this paper) shows that this system may at most be semi-globally practically stable in the expectation of the second moment, thereby choosing $|x_0 - x^{*}|$ to be sufficiently large will ensure divergence of the trajectories for the chosen parameters.

\begin{figure}[!t]
    \centering
    \includegraphics[width=8cm]{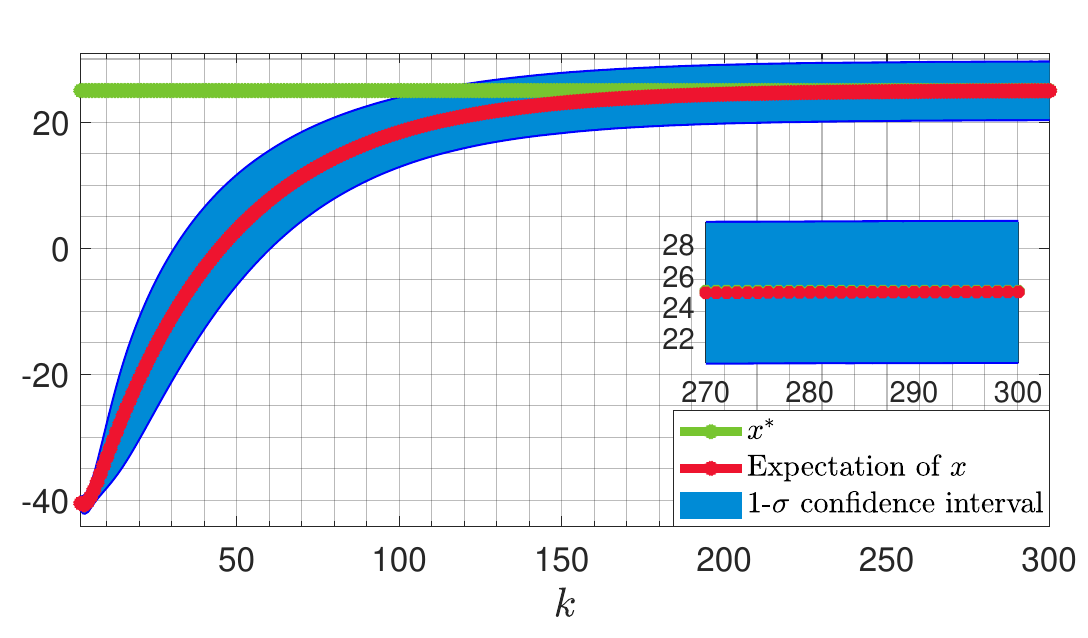}
    \caption{Simulation result for system \eqref{sys_def} with objective function \eqref{one_dim_q} and no adaptation of the step size}
    \label{fig:quadratic_2.2}
\end{figure}

In addition, we have simulated system \eqref{dtses} in the form of 
\begin{align}
    \label{first_order_system_sim}
    x_{k+1} = x_{k} - h_{k-1}(J(x_k) - J(x_{k-1})) + g_k,
\end{align}
where $h_k := 10^{-3}w_k$ and $g_k := 9w_k$. The results are shown on \Cref{fig:quadratic_2.3}, where we observe the steady state variance to be significantly larger than in \Cref{fig:quadratic_2.2}, while in this case, according to our analysis, only local convergence may be proven. If a time-dependent map $g_k$ were to be used, for example $g_k := \frac{9}{\sqrt{k}}w_k$, the steady state variance could be reduced significantly at the cost of convergence speed.

\begin{figure}[!t]
    \centering
    \includegraphics[width=8cm]{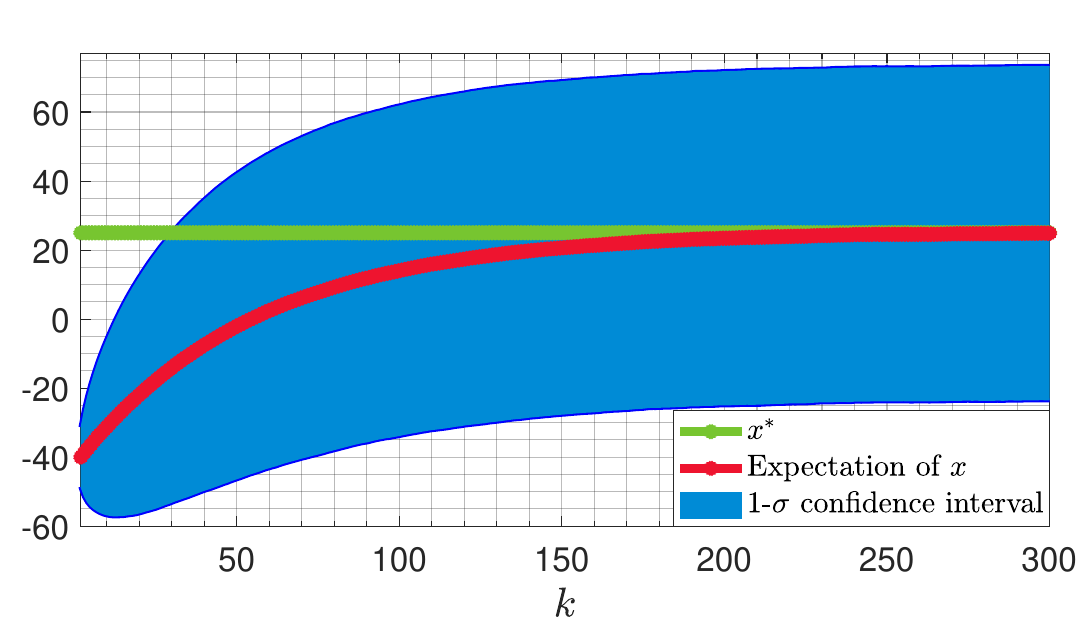}
    \caption{Simulation result for system \eqref{first_order_system_sim} with objective function \eqref{one_dim_q}}
    \label{fig:quadratic_2.3}
\end{figure}

We henceforth present simulations where the maps $\hat{h}, \hat{g}$ comply with \Cref{assump:ghw} \ref{assump:ghw:1}, i.e. an adaptive step size has been used. We next present the numerical results for system \eqref{sys_def} for the convex objective function \eqref{one_dim_convex}, where the initial point is chosen as $x_0 = -240$ and $x^{*} = 350$, and $y_0$ is chosen identically as for the simulation presented in \Cref{fig:quadratic_1}. The system parameters are chosen as
\begin{align*}
    \rho = 0.55, && \chi = 100, && \psi = 0.36, && \beta = 0.5, && \varepsilon = 10^{-7},
\end{align*}
and $h,g$ are chosen as for the simulations in \Cref{fig:quadratic_1}. The results are shown on \Cref{fig:quadratic_3}, and they show exact convergence of the expectation, however only a practical convergence for the variance. It remains to be mentioned that we have not observed a divergence result by changing the distance between the initial condition $x_0$ and the global minimizer $x^{*}$, similar as in the case of the quadratic objective optimization problem in Figures \ref{fig:quadratic_1} and \ref{fig:quadratic_2}.
\begin{figure}[!t]
    \centering
    \includegraphics[width=8cm]{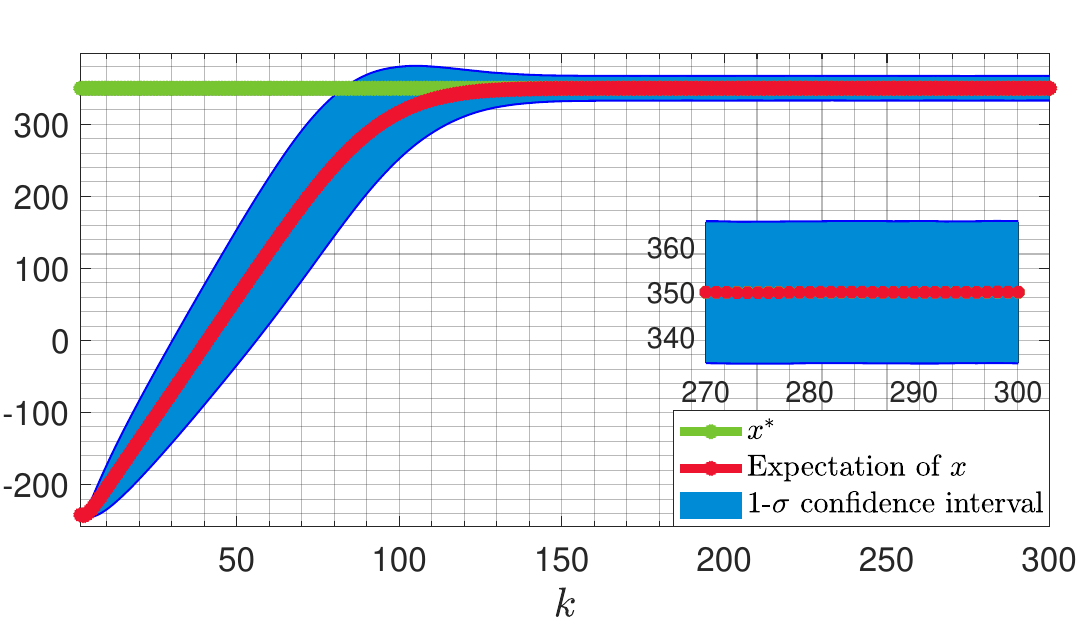}
    \caption{Simulation result for system \eqref{sys_def} with objective function \eqref{one_dim_convex}}
    \label{fig:quadratic_3}
\end{figure}

Finally, we present our numerical results for multidimensional optimization. 
In this simulation setup, we have~${x_k \in \R^n}$ and $y_k \in \R^n$ at any $k \geq 0$, where $n \in \N$. For the objective function \eqref{multi_dim_q} the implemented multidimensional system has the form
\begin{subequations}
    \label{system_multidim}
    \begin{align}
        x_{k+1} &= x_k - \rho y_k + \sqrt{\psi} \Omega_k(\mathrm{abs}(y_k) + \varepsilon\mathbf{1}_n)\\
        y_{k+1} &= (1-\beta)y_k + \sqrt{\chi}\lVert \mathrm{abs}(y_{k-1}) + \varepsilon\mathbf{1}_n\rVert_2^{-2}  \Omega_{k-1} \notag \\ & \times(\mathrm{abs}(y_{k-1}) + \varepsilon\mathbf{1}_n)(J(x_k) - J(x_{k-1})),
    \end{align}
\end{subequations}
where $\mathrm{abs}(y)$ denote the element-wise absolute values of a vector $y \in \R^m$, $\mathbf{1}_n$ is an $n$ dimensional vector of ones, and $\Omega_k$ for all $k \geq 0$ is a diagonal matrix of i.i.d. random variables given as
\begin{align*}
    \Omega_k &= \begin{bmatrix}
        w_k^{(1)} & 0 & \dots & 0 \\
        0 & w_k^{(2)} & \dots & 0 \\
        \vdots & \vdots & \ddots & \vdots \\
        0 & 0 & \dots & w_k^{(n)}
    \end{bmatrix}, \\ 
    w_k^{(j)} &\distas{} \D \hspace{0.25cm} \forall j \in \{1,2,\dots,n\}.
\end{align*}
The objective function \eqref{multi_dim_q} is defined by $n=3$,
\begin{align*}
    x^{*} &= \begin{bmatrix}
         5.2 \\
         1.23 \\
        -3.2 
    \end{bmatrix} \cdot 10^{5}, \hspace{0.5cm}
    H  = \begin{bmatrix}
        0.7 & 0.1 & 0.2 \\
        0.3 & 0.4 & 0.3 \\
        0.4 & 0 & 0.5
    \end{bmatrix}.
\end{align*}
The matrix $H$ is positive definite and thus \eqref{multi_dim_q} is strongly convex. The algorithm parameters are chosen as
\begin{align*}
    \rho &= 0.25,& \chi &= 0.2025, & \psi &= 0.01, & \beta &= 0.93,
\end{align*}
and $\varepsilon$ as in the previous simulations. The simulation results are shown on \Cref{fig:quadratic_5}, where we observe that the expectation of the sample trajectories of $x$ converges to the global minimizer $x^{*}$. 
Hence, this simulation indicates that similar convergence results to the one dimensional case also hold in higher dimensions. 
We note that the convergence rate is slower than in the one dimensional case, but also we did not perform a parameter tuning.
\begin{figure}[!t]
    \centering
    \includegraphics[width=8cm]{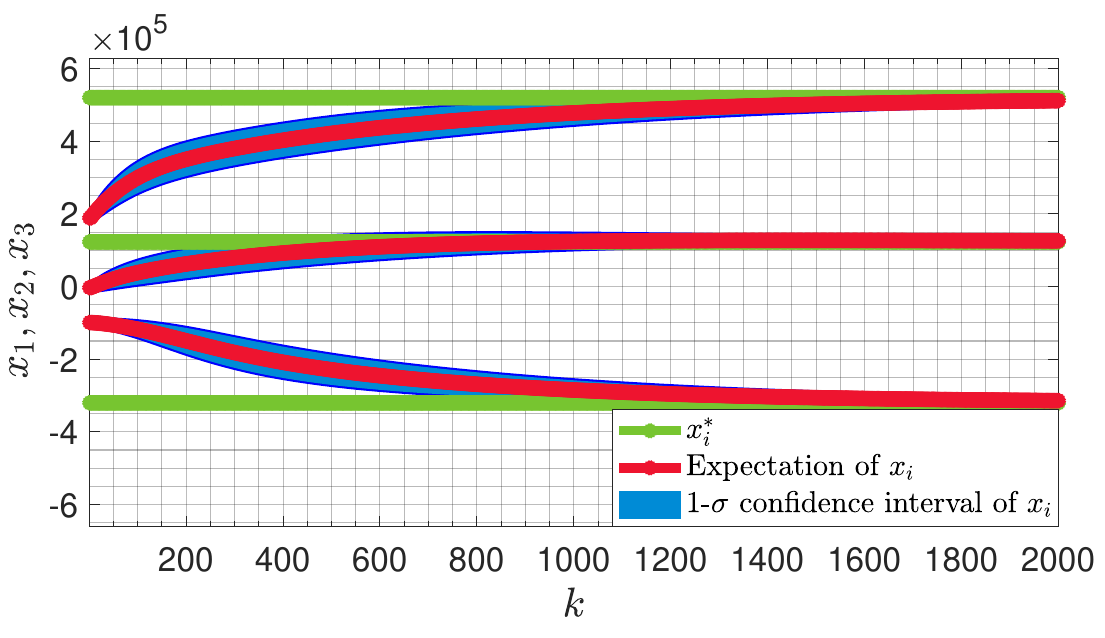}
    \caption{Simulation result for system \eqref{system_multidim} with objective function \eqref{multi_dim_q} }
    \label{fig:quadratic_5}
\end{figure}

\section{Conclusions and Outlook}
\label{section_VI}
In this paper, we have proposed a novel discrete-time stochastic optimization and ES algorithm, initially motivated by a class of nonlinear stochastic delay-differential equations. 
In particular, the two underlying ideas and constructions of the proposed algorithm were 
\begin{enumerate}[label = \roman*)]
    \item to replace deterministic exploration sequences by stochastic (time-delayed) ones, and 
    \item to introduce an adaptive step size mechanism (in terms of $\hat h, \hat g$ 
    and a (moment-like) dynamics $y_k$).
\end{enumerate}
For the case of a one-dimensional unconstrained quadratic optimization problem, we have proven convergence in expectation to the global optimizer, and convergence of the variance to a user-determined $\varepsilon$-neighbourhood of zero for a suitable set of system parameters. 
A convergence proof for one-dimensional optimization problems is rather a special case but allows to keep technical
steps rather simple (though lengthy) and helps to clarify ideas and proof steps. It is not uncommon in the literature to consider one dimensional cases, as for example, \cite{manzie_krstic2007discrete, zhu_fridman2021continuous, choi_krstic2002discrete_time_extremum} have also considered a one-dimensional optimization problem.
Further, we have presented several simulations, amongst which ones who indicate that the presented algorithm may be used for a broader class of optimization problems. In particular, we have shown the simulation results for a one-dimensional non-quadratic convex optimization problem, a one-dimensional "locally" quadratic problem, and with a suitable adaptation for a multidimensional quadratic optimization problem.

For future research, we believe it would be interesting to study other distributions for the exploration sequences, for instance zero-mean Gaussian random variables. 
One advantage of the proposed stochastic algorithm in the multi-dimensional case, for example in comparison to the deterministic algorithm proposed in \cite{feiling_ali_christian2021gradient},
is that the exploration sequences can be easily
designed by simply choosing i.i.d. random variables for each dimension (decision variable), as already indicated in our simulation results. 
Hence, it is an important future research goal to establish a proof of the multidimensional case, which is currently investigated. 
Other future research directions include the investigation of quadratic and convex optimization problems with linear inequality constraints or unconstrained optimization with strongly convex objective functions, and comparing the convergence of the sequences generated by \eqref{sys_def} to those of existing zero-order optimization methods \cite{ZO_optimization_2020primer}. 

\newpage
\bibliographystyle{unsrt}  
\bibliography{references.bib}

\newpage
\section{Appendix}

\label{section_appendix}

\subsection{Auxiliary Results}

In this section, we prove several properties of \eqref{sys_def}
including \Cref{theorem:convergence_means_square}. 
Regarding notation, in the following we denote by $A^{(i,*)}$ and~$b^{(i)}$ the $i$-th row of matrix $A$ and $i$-th element of vector~$b$, respectively.

The first result is interesting by its own and used in
the proof of  \Cref{theorem:convergence_means_square} to ensure global exponential practical convergence of the variance to $0$ of almost every trajectory of system \eqref{sys_def} for some small~$\varepsilon > 0$. One may observe that it is related to the Comparison Lemma, i.e. Lemma 3.4 in \cite{khalil2002nonlinear}, at least in discrete-time, however to the best of our knowledge, existing similar results,  for instance Corollary 13.8 in \cite{haddad2008nonlinear}, do not include this results.

\begin{proposition}
    \label{lemma:matrix_upper_lower_bounds}
    Let $\eta_0 \in \R^{+}$ and $\theta_k \in \R^n$ for all $k \geq 0$. Suppose that $\theta_0$ is arbitrary and moreover at every $k$ it holds element-wise that
    \begin{align*}
        (A + \eta Q_1)\theta_k + \eta b_1 \leq \theta_{k+1} \leq (A + \eta Q_2)\theta_k + \eta b_2,
    \end{align*}
    where $\eta \in (0,\eta_0]$, $A, Q_1, Q_2 \in \R^{n \times n}$ and~$ b_1, b_2 \in \R^n$, and all the eigenvalues of~$A$ are inside the unit circle. If $\eta_0$ is sufficiently small, then for any $\eta \in (0, \eta_0]$ the sequence $\theta_k$ converges exponentially to an $\eta$-neighbourhood of~$0$.
\end{proposition}

\begin{proof}[Proof of \Cref{lemma:matrix_upper_lower_bounds}]
    The proof contains two steps. In the first step we use quadratic Lyapunov functions in order to show that for some sufficiently small $\eta > 0$, the first order difference of the Lyapunov functions at every $k \geq 0$ is upper bounded by the sum of a negative definite quadratic function and a~$\eta$-dependent disturbance term. In exact, in 
    the first step, for some $P \succ 0$ and~${V(\theta) := \theta^{\top}P\theta}$ we  show that at any $k \geq 0$
    \begin{align*}
        V(\theta_{k+1}) - V(\theta_k) \leq -\alpha V(\theta_k) + \eta R(\theta_k),
    \end{align*}
    where $R: \R^n \to \R$ is linear in $\theta_k$ and $\alpha > 0$.
    We assume that $P\succ 0$ is a symmetric matrix for which $A^\top P A - P \prec 0$  is satisfied, as per assumption we have the matrix $A$ to have all its eigenvalues inside the unit circle.
    
    Consider an arbitrary $k \geq 0$. We first want to find the upper bound of Lyapunov function difference for any $\theta_{k+1}$, which as assumed, is bounded as follows
    \begin{align*}
        (A + \eta Q_1)\theta_k + \eta b_1 \leq \theta_{k+1} \leq (A + \eta Q_2)\theta_k + \eta b_2.
    \end{align*}
    Now, define the terms ${l_k:= (A + \eta Q_1)\theta_k + \eta b_1}$ and~${q_k := (A + \eta Q_2)\theta_k + \eta b_2}$. Then,
    \begin{align*}
        \Delta V(\theta_{k+1},\theta_k) &:= V(\theta_{k+1}) - V(\theta_k) \\ &\leq \underset{\theta_{k+1} \in [l_k,q_k]}{\mathrm{max}}\big(\theta_{k+1}^{\top}P\theta_{k+1}\big) - \theta_k^{\top}P\theta_{k},
    \end{align*}
    where it holds element-wise $l_k \leq \theta_{k+1} \leq q_k$, and for the proof we first aim to find the maximum value of the term $\theta_{k+1}^{\top} P \theta_{k+1}$, in order to thereafter analyze the upper bound of the Lyapunov difference term. For simplicity, we denote ${\Delta V := \Delta V(\theta_{k+1},\theta_k)}$. To find the maximum, we use the parameterization ${\theta_{k+1}:= Tl_k + (I - T)q_k}$, where $I$ is the identity matrix and 
    \begin{align*}
        T &: = \begin{bmatrix}
            t_1 & 0 & \dots & 0 \\
            0 & t_2 & \dots & 0 \\
            \vdots & \vdots & \ddots & \vdots \\
            0 & 0 & \dots & t_n
        \end{bmatrix},
    \end{align*}
    where $t_j \in [0,1], j \in \{1,2,\dots,n\}$. We observe that the matrices $T$ and $I - T$ are at least positive semi-definite. Moreover, for the sake of clarity we define ${P_{11}:=T^{\top}PT}$, ${P_{12}:=T^{\top}P(I-T)}$ and ${P_{22}:=(I-T)^{\top}P(I-T)}$. 

    Next, we aim to find the maximum for the term $\theta_{k+1}^{\top} P \theta_{k+1}$, and by using that $\theta_{k+1} = Tl_k + (I - T)q_k$ we get
    \begin{align*}
        \theta_{k+1}^{\top}P\theta_{k+1} &= l_k^{\top}P_{11}l_k + 2l_k^{\top}P_{12}q_k + q_k^{\top}P_{22}q_k,
    \end{align*}
    where we have used $P_{12} = P_{12}^{\top}$ to get 
    \begin{align*}
        l_k^{\top}P_{12}q_k + q_k^{\top}P_{12}^{\top}l_k = 2l_k^{\top}P_{12}q_k,
    \end{align*}
    and this represents a property which we use extensively throughout the proof.
    We further observe that, as~${P \succ 0}$ and~${T\succeq0}$, ${I - T \succeq 0}$, it follows that ${P_{11}\succeq 0}$,${P_{12}\succeq 0}$ and ${P_{22} \succeq 0}$.
    Next, we define $${w_k : = 2l_k^{\top}P_{12}q_k},$$
    and observe that trivially $w_k \leq |w_k|$. From Cauchy-Schwartz's inequality (see e.g. \cite[Theorem 7.7.11]{horn2012matrix}) we get
    \begin{align*}
        |w_k| = 2|l_k^{\top}P_{12}q_k| &\leq 2\sqrt{l_k^{\top}P_{12}l_k} \sqrt{q_k^{\top}P_{12}q_k},
    \end{align*}
    and by using the maximum of both expressions, we further get $w_k \leq |w_k|\leq W $ where
    \begin{align*}
        W := 2\mathrm{max}{(l_k^{\top}P_{12}l_k,} q_k^{\top}P_{12}q_k).
    \end{align*}
    As we have now an upper bound for the mixed term, it then remains to represent both $l_k$ and $q_k$ as functions of~$\theta_k$. By using the upper bound $W$ for the mixed terms, for $\theta_{k+1}^{\top} P \theta_{k+1}$ we have
    \begin{align}
        \label{proof:main_lemma:1}
        \theta_{k+1}^{\top}P\theta_{k+1} &\leq l_k^{\top}P_{11}l_k + 2\mathrm{max}{(l_k^{\top}P_{12}l_k,} q_k^{\top}P_{12}q_k) \notag \\ &+ q_k^{\top}P_{22}q_k.
    \end{align}
    We shall now consider simultaneously both cases for the middle term. Let $i \in \{1,2\}$, and if $W = l_k^{\top}P_{12}l_k$ let $i = 1$, and if $W = q_k^{\top}P_{12}q_k$ let $i = 2$.
    In the following we use the property
    \begin{align*}
        P_{11} + P_{12} + P_{12}^{\top} + P_{22} &= P,
    \end{align*}
    which we leave to the reader to verify, as well as ${P_{12} = P_{12}^{\top}}$ to obtain
    \begin{align*}
        \theta_k^{\top}A^{\top}(P_{11} + 2P_{12} + P_{22})A\theta_k = \theta_k^{\top}A^{\top}PA\theta_k. 
    \end{align*} 
    Next, by substituting the terms $l_k= (A + \eta Q_1)\theta_k + \eta b_1$ and~${q_k = (A + \eta Q_2)\theta_k + \eta b_2}$ in \eqref{proof:main_lemma:1}, and by using the aforementioned properties as well as several algebraic transformations, we obtain
    \begin{align*}
        \theta_{k+1}^{\top}P \theta_{k+1} &\leq \theta_k^{\top}(A^{\top}PA + \eta \Bar{Q}(i))\theta_k + 2\eta \Bar{M}(i)\theta_k \\ & + \eta^2 \Bar{b}(i),
    \end{align*}
    where
    \begin{align*}
        \Bar{Q}(i) &:= 2Q_1^{\top}P_{11}A + \eta Q_1^{\top}P_{11}Q_1 +2Q_2^{\top}P_{22}A \\ &+ \eta Q_2^{\top}P_{22}Q_2 + 4Q_i^{\top}P_{12}A + 2\eta Q_i^{\top}P_{12}Q_i, \\
        \Bar{M}(i)&:= b_1^{\top}P_{11}(A + \eta Q_1) +  b_2^{\top}P_{22}(A + \eta Q_2) \\ &+  2b_i^{\top}P_{12}(A + \eta Q_i), \\
        \Bar{b}(i) &:= b_1^{\top}P_{11}b_1 + b_2^{\top}P_{22}b_2 + 2b_i^{\top}P_{12}b_i.
    \end{align*}
    
    In the following step we analyze the upper bound of the Lyapunov function difference. For it, we now get the upper bound 
    \begin{align*}
       \Delta V &\leq \theta_k^{\top}(A^{\top}PA - P + \eta \Bar{Q}(i))\theta_k + 2\eta \Bar{M}(i)\theta_k + \eta^2 \Bar{b}(i) \\&=: \theta_k^{\top}(A^{\top}PA - P + \eta \Bar{Q})\theta_k + \eta R(\theta_k,i).
    \end{align*}
    For simplicity denote $B(\eta,i): = A^{\top}P A - P + \eta \Bar{Q}(i)$, which has to be negative definite for any $i$ to ensure stability, and by the assumption on the eigenvalues of $A$ we have~${A^{\top}P A - P \prec 0}$. As eigenvalue perturbation is continuous, it follows that there exists some interval of values for~$\eta > 0$, such that the quadratic term for both $i = 1$ and~$ i = 2$ is negative definite. Then, it is possible to find some $\alpha(\eta) > 0$, such that
    \begin{align*}
        V(\theta_{k+1}) - V(\theta_k) \leq -\alpha(\eta) V(\theta_k) + \eta R(\theta_k).
    \end{align*}

    In the second step, we use that $B(\eta, i) \prec 0 $ for a sufficiently small $\eta > 0$ for any $i$, and to describe the idea, consider some set ${\mathcal{V} :=\{ z \in \R^n : \lVert z \rVert_2^2 \geq \nu_0\}}$, where~${\nu_0 > 0}$ is a sufficiently large constant. Then, for any $\theta_k \in \mathcal{V}$ it follows that $|\theta_k^{\top} B(\eta,i)\theta_k| > |\eta R(\theta_k, i)|$, i.e. the negative definite quadratic term dominates the linear term, hence $\theta_k^{\top}B(\eta,i)\theta_k + \eta R(\theta_k,i) < 0$. As this term is negative for any $\theta_k \in \mathcal{V}$, it is then possible to find some sufficiently small constant $\alpha_0 > 0$, such that for all $\theta_k \in \mathcal{V}$ it holds
    \begin{align*}
        \theta_k^{\top}B(\eta,i)\theta_k + \eta R(\theta_k,i) \leq -\alpha_0\theta_k^{\top}B(\eta,i)\theta_k < 0.
    \end{align*}
    Therefore it is possible to claim that for any $\eta > 0$ for which ${B(\eta, i) \prec 0}$ there exist ${\alpha_{\text{min}}(\eta) > 0}$ and ${\beta(\eta)> 0}$, such that for all ${\theta_k \in \mathcal{S}(\eta) :=\{ z \in \R^n : \lVert z \rVert_2^2 \geq \beta(\eta)\}}$
    \begin{align*}
        \Delta V &\leq \theta_k^{\top}B(\eta,i)\theta_k + \eta R(\theta_k,i)  \\ &\leq - \alpha_{\text{min}}(\eta)\theta_k^{\top}P\theta_k,
    \end{align*}
    i.e. a negative definite quadratic term dominates the sum of the negative definite quadratic term and the linear term. This proves global exponential convergence, i.e. stability, for the radially unbounded set $\mathcal{S}(\eta)$. It remains to show that for~${\eta \to 0}$ it follows that $\mathcal{S}(\eta) \to \R^n$. Observe that
    \begin{align*}
        &\lim\limits_{\eta \to 0} \theta_k^{\top}B(\eta,i)\theta_k + \eta R(\theta_k,i) =  \theta_k^{\top}(A^{\top}PA - P)\theta_k,
    \end{align*}
    which is negative definite for all $\theta_k \in \R^n$, and as this represents an upper bound to the Lyapunov function difference it follows that we proved global exponential practical stability.
\end{proof}

The following technical results are necessary to perform the analysis of the dynamics of system \eqref{sys_def} in expectation and in expectation of the second moment. We start out with a (Taylor) expansion of the objective function to be able to use the influence of the time delayed random variable, and then continue with several general properties used throughout the proofs presented in this paper.

\begin{lemma}
    \label{proposition:quadratic_decomposition}
    Consider system \eqref{sys_def} with \Cref{assump:ghw}, and suppose \Cref{assump:quadratic} holds. Then at every $k \geq 1$
    \begin{align}
        \label{proposition:DELTA}
        J(x_k)  &= J(x_{k-1}) +  \mu \Delta_{k-1}, \notag \\
        \Delta_{k-1} &:= ((x_{k-1} - x^{*}) - \rho y_{k-1})(|y_{k-1}| + \varepsilon)g_{k-1} \notag \\
        &- \rho (x_{k-1} - x^{*})y_{k-1} + \frac{\rho^2}{2}y_{k-1}^2 \notag \\ &+\frac{g_{k-1}^2}{2}(|y_{k-1}|  + \varepsilon)^2. \tag{L.1}
    \end{align}
\end{lemma}
\begin{proof}[Proof of \Cref{proposition:quadratic_decomposition}]
    For the proof we use \Cref{assump:quadratic} and a second order Taylor expansion. No higher order terms exist due to the aforementioned assumption.
    By expanding $J(x_k)$ in the neighbourhood of $x_{k-1}$ we get
    \begin{align*}
        J(x_{k}) &= J(x_{k-1}) + \nabla J(x_{k-1}) (x_{k} - x_{k-1})\\ & + \frac{1}{2}\nabla^2 J(x_{k-1})(x_{k} - x_{k-1})^2.
    \end{align*}
    Further, according to \Cref{assump:quadratic} we have $$\nabla J(x_{k-1}) = \mu(x_{k-1} - x^{*})$$ and $\nabla^2 J(x_{k-1}) = \mu$, and from \eqref{sys_def} we get $$x_{k} - x_{k-1} = -\rho y_{k-1} + \hat{g}_{k-1}.$$ We then substitute $\hat{g}_{k-1} = (|y_{k-1}| + \varepsilon)g_{k-1}$, and get the relation
    \begin{align*}
        J(x_{k}) &= J(x_{k-1}) + \mu\Tilde{x}_{k-1}(-\rho y_{k-1} + (|y_{k-1}| + \varepsilon)g_{k-1})\\ & + \frac{1}{2}\mu (\rho^2 y_{k-1}^2 + (|y_{k-1}| + \varepsilon)^2g_{k-1}^2) \\ &- \mu\rho y_{k-1}(|y_{k-1}| + \varepsilon)g_{k-1}).
    \end{align*}
   By reorganizing the terms we obtain the required relation. 
\end{proof}

\begin{lemma}
    \label{proposition:general_properties}
    Consider system \eqref{sys_def} with \Cref{assump:ghw}. Then, for any map $f: \R^3 \to \R$, $m,p \in \N\cup \{0\}$ at any $k \geq 0$, it follows that
    \begin{align}
        \label{proposition:general_properties:1} &\E\left[\hat{h}_k\hat{g}_k\right] = \E[h_kg_k] =: \gamma > 0, \tag{\textit{L.2.a}} \\
        \label{proposition:general_properties:5} &\E\left[f(x_k,y_k,y_{k+1})h_k^mg_k^p\right] =   \E\left[f(x_k,y_k,y_{k+1})\right] \E\left[h_k^mg_k^p\right], \tag{\textit{L.2.b}} \\
        \label{proposition:general_properties:2} &\E\left[\hat{h}_k^m\hat{g}_k^p\right] = \E\left[h_k^mg_k^p\right] = 0, \text{ if } (m+p) \text{ odd}, \tag{\textit{L.2.c}} \\
        \label{proposition:general_properties:3} &\E\left[h_k^mg_k^{m+p}\right] = \gamma^m\psi^p = \chi^{\frac{m}{2}}\psi^{\frac{m}{2} + p}, \text{ if } p \text{ even}. \tag{\textit{L.2.d}}
    \end{align}
\end{lemma}
\begin{proof}[Proof of \Cref{proposition:general_properties}]
    The proof contains four sections, one for each of the relations. We start with \eqref{proposition:general_properties:1}, where by \Cref{assump:ghw}~\ref{assump:ghw:1}, $$\E\left[\hat{h}_k\hat{g}_k\right] =\E\left[\frac{h_k}{|y_k| + \varepsilon}g_k(|y_k| + \varepsilon)\right] = \E[h_kg_k]$$  for all $y_k \in \R$, and as $\mathrm{sign}(\hat{g}_k)=\mathrm{sign}(\hat{h}_k)$ for all $w_k \in \D$ it follows also that $h_kg_k > 0 $ for all $k \geq 0$. As $\D$ has no zero-valued elements, the first claim is proven.

    We next consider \eqref{proposition:general_properties:5}. Observe from \eqref{sys_def} that at time step $k+1$, $x_{k+1}$ depends on the term $\hat{g}_k$, which is a function of $w_k$, and moreover on $y_k$, which is in turn influenced by~$w_{k-2}$ via $\hat{h}_{k-2}$. However, $x_{k+1}$ is stochastically independent from $w_{k+1}$, as $w_{k+1}$ is drawn randomly from $\D$ at time step~$k+1$. The same is true for $y_{k+1}$ and~$w_{k+1}$ by analogy. Therefore, by considering any time step $k \geq 0$, we have $x_k$ stochastically independent from $w_k$, and $y_k$ stochastically independent from $w_k$. It remains to show that $w_k$ is independent from $y_{k+1}$, and for $k = 0$ the proof is trivial, as $y_1$ is an initialization parameter, therefore independent from $w_{0}$. For $k\geq 0$ it is necessary to show~$\E[w_ky_{k+1}] = \E[w_k]\E[y_{k+1}]$.
    By expanding $y_{k+1}$ using \eqref{sys_def}, we get $$w_ky_{k+1} = w_k(1-\beta)y_k + w_k\hat{h}_{k-1}\left(J(x_k) - J(x_{k-1})\right).$$ As shown previously, the pairs $(w_k, y_k)$ and $(w_k, x_k)$ are stochastically independent, and as $w_k$ is sampled from $\D$ at time step $k$ it follows that $(w_k, x_{k-1})$ are stochastically independent. Thus,
    \begin{align*}
        \E[w_ky_{k+1}] &= \E[w_k](1-\beta)\E[y_k] \\ &+\E[w_k]\E\left[\hat{h}_{k-1}\left(J(x_k) - J(x_{k-1})\right)\right] \\ &= \E[w_k]\E[y_{k+1}].
    \end{align*}
    Hence, for any map $f(x_k,y_k,y_{k+1})$ it holds that it is stochastically independent from $w_k$ and by extension  independent of $h(w_k),g(w_k)$, thereby proving the claim.
    
    Further, we consider \eqref{proposition:general_properties:2}, where by using the definition of $\hat{h}_k, \hat{g}_k$ given in \Cref{assump:ghw} \ref{assump:ghw:1}, and using the stochastic independence of the pair $(y_k,w_k)$ as stated in   \eqref{proposition:general_properties:5}, we obtain that~$\E\left[\hat{h}_k^m\hat{g}_k^p\right] = \E\left[(|y_k| + \varepsilon)^{p-m}\right] \E\left[h_k^mg_k^p\right]$. We know from \Cref{assump:ghw}~\ref{assump:ghw:1} that $h(\w),g(\w)$ are odd in $\w$, and as~${w_k \in \{-\w, \w\}}$, we have 
    \begin{align*}
        \E\left[h_k^mg_k^p\right] &= \frac{1}{2}\left( h(\w)^mg(\w)^p + h(-\w)^mg(-\w)^p\right) \\ &= \frac{1}{2}\left( h(\w)^mg(\w)^p + \left((-h(\w)\right)^m\left(-g(\w)\right)^p\right).
    \end{align*}
    Note once more that $ h(\w)g(\w) > 0 $ for all $\w \in \R^{+}$. It is thus clear that if $(-1)^{m+p} = -1$, the given relation is identical to $0$.  
    Thus, we conclude that the term $ \E\left[h_k^mg_k^p\right]$ is zero if and only if $m+p$ is odd.
    
     We lastly consider property \eqref{proposition:general_properties:3}. By \Cref{assump:ghw} \ref{assump:ghw:2} we have 
    \begin{align*}
        \E\left[h_k^mg_k^{m+p}\right] &= \frac{1}{2}\Big(h(\w)^mg(\w)^mg(\w)^p \\ &+ h(-\w)^mg(-\w)^mg(-\w)^p\Big).
    \end{align*}
    As $h(\w),g(\w)$ are odd in $\w$, it follows that $${h(\w)^mg(\w)^m = h(-\w)^mg(-\w)^m > 0}$$ for any $m \in \N \cup \{0\}$, ${\w \neq 0}$, and moreover it holds that $\frac{h(\w) - h(-\w)}{2} = h(\w)$, $\frac{g(\w) - g(-\w)}{2} = g(\w)$. We then rewrite the expression as 
    \begin{align*}
        \E\left[h_k^mg_k^{m+p}\right] &= \frac{1}{2}h(\w)^mg(\w)^m\Big(g(\w)^p + g(-\w)^p\Big) \\
        &=\left(\frac{h(\w) - h(-\w)}{2}\right)^{m}\\ &\times \left(\frac{g(\w) - g(-\w)}{2}\right)^{m} \psi^p,
    \end{align*}
    as $\frac{1}{2}\left(g(\w)^p + g(-\w)^p\right) = \E\left[g_k^p\right] = \psi^p$ by \Cref{assump:ghw} \ref{assump:ghw:2}.
    W.l.o.g. we only consider $h(\w) - h(-\w)$, and start out for~$m$ being even. Then, 
    \begin{align*}
        \left(\frac{h(\w) - h(-\w)}{2}\right)^{2} = \frac{h(\w)^2 + h(-\w)^2 - 2h(\w)h(-\w)}{4}
    \end{align*}
    where we know that as $h$ is odd and by \Cref{assump:ghw} \ref{assump:ghw:2} we have $$\frac{1}{2}\frac{h(\w)^2 + h(-\w)^2}{2} = \frac{1}{2}\E[h_k^2] = \frac{1}{2}\chi,$$ and moreover $$\frac{-2h(\w)h(-\w)}{4} = \frac{h(\w)^2}{2} = \frac{1}{2}\frac{h(\w)^2 + h(-\w)^2}{2} = \frac{1}{2}\chi,$$ where we used $h(\w)^2 = h(-\w)^2$. As 
    \begin{align}
        \label{proof:proposition:1:eq:1}
        h(\w)^m = \left(\left(\frac{h(\w) - h(-\w)}{2}\right)^{2}\right)^{\frac{m}{2}}
    \end{align}
    we get $$h(\w)^m = \chi^{\frac{m}{2}},$$ with the proof of $g(\w)^m$ following analogously. Consider now $m$ to be odd, and therefore $m-1$ to be even. Then we get 
    \begin{align*}
        h(\w)^{m}g(\w)^{m} &= h(\w)^{m-1}g(\w)^{m-1}h(\w)g(\w) \\ &= (\chi\psi)^{\frac{m-1}{2}}h(\w)g(\w).
    \end{align*}
    Observe that $$h(\w)g(\w) = \mathrm{sign}(h(\w))\sqrt{h(\w)^2} \mathrm{sign}(g(\w))\sqrt{g(\w)^2}.$$ From \Cref{assump:ghw} \ref{assump:ghw:1} we have $\mathrm{sign}(h(\w)) = \mathrm{sign}(h(\w))$ for any $\w \in \R$, therefore their product always equals $1$. Hence, we get $$h(\w)g(\w) = \sqrt{h(\w)^2} \sqrt{g(\w)^2},$$ for which we use \eqref{proof:proposition:1:eq:1} to get
    \begin{align*}
        h(\w)g(\w) = \sqrt{h(\w)^2} \sqrt{g(\w)^2} = \sqrt{\chi}\sqrt{\psi},
    \end{align*}
    thereby concluding the proof.
\end{proof}

\begin{rmk}
    \label{remark:gamma}
    In the forthcoming analysis we use both notations of $\gamma$ and $\sqrt{\chi\psi}$, which by \Cref{proposition:general_properties} \ref{proposition:general_properties:3} with~${m=1}$,~$p = 0$ are equivalent.
\end{rmk}

\begin{lemma}
    \label{proposition:expectation_h_k-1_delta_k-1}
    Consider system \eqref{sys_def} with \Cref{assump:ghw}, and suppose \Cref{assump:quadratic} holds. Then, for any function ${\bar{f}: \R^3 \to \R}$ at any $k \geq 1$
    \begin{align*}
        &\E\left[\frac{\mu h_{k-1}\bar{f}(x_{k-1},y_{k-1},y_k)}{|y_{k-1}| + \varepsilon}\Delta_{k-1} \right] \\  = & \mu \gamma \E\Big[\bar{f}(x_{k-1},y_{k-1},y_k) (x_{k-1} - \rho y_{k-1})\Big],
    \end{align*}
    where $\Delta_{k-1}$ is as defined in \eqref{proposition:DELTA}.
\end{lemma}
\begin{proof}[Proof of \Cref{proposition:expectation_h_k-1_delta_k-1}]
    For the proof we use the stochastic independence of any of the terms $x_{k-1},y_{k-1},y_k$ from $w_{k-1}$ as stated in   \eqref{proposition:general_properties:5}. Denote for simplicity $ {\bar{f}(\cdot):=\bar{f}(x_{k-1},y_{k-1},y_k)}$ and
    \begin{align*}
        z_{k-1} := \frac{\mu h_{k-1}\bar{f}(x_{k-1},y_{k-1},y_k)}{|y_{k-1}| + \varepsilon}\Delta_{k-1}.
    \end{align*}
    By expanding $\Delta_{k-1}$ and performing elementary algebraic transformations, we get
    \begin{align*}
        z_{k-1} &= \mu  \bar{f}(\cdot)\Bigg(h_{k-1}g_{k-1}(x_{k-1} - \rho y_{k-1}) \\ &+ \frac{h_{k-1}}{2(|y_{k-1}| + \varepsilon)}\Big( \rho^2 y_{k-1}^2 + g_{k-1}^2(|y_{k-1}| + \varepsilon)^2 \\ &- 2\rho x_{k-1}y_{k-1}\Big) \Bigg).
    \end{align*}
    Moreover, due to the above mentioned stochastic independence property following from   \eqref{proposition:general_properties:5}, and further employing   \eqref{proposition:general_properties:1} we have 
    \begin{align*}
        &\E\left[\mu \bar{f}(\cdot)h_{k-1}g_{k-1}(x_{k-1} - \rho y_{k-1})\right] \\ \overset{\eqref{proposition:general_properties:5}}{=}&\E\left[h_{k-1}g_{k-1}\right]\E\left[\mu \bar{f}(\cdot)(x_{k-1} - \rho y_{k-1})\right]  \\\overset{\eqref{proposition:general_properties:1}}{=}& \mu\gamma\E\Big[ \bar{f}(\cdot) (x_{k-1} - \rho y_{k-1}) \Big].
    \end{align*}
    Next, we consider the two remaining terms 
    \begin{align*}
        t_1:= &\mu\E\left[h_{k-1}\frac{ \bar{f}(\cdot)(\rho^2y_{k-1}^2 -2\rho x_{k-1}y_{k-1})}{2(|y_{k-1}| + \varepsilon)}\right] \\
        t_2: =&\mu\E\left[h_{k-1}g_{k-1}^2\frac{ \bar{f}(\cdot)(|y_{k-1}| + \varepsilon)}{2}\right],
    \end{align*}
    where from   \eqref{proposition:general_properties:5} we observe once more that the above mentioned stochastic independence property holds, and thereby
    \begin{align*}
        t_1\overset{\eqref{proposition:general_properties:5}}{=} &\mu\E\left[h_{k-1}\right]\left[\frac{ \bar{f}(\cdot)(\rho^2y_{k-1}^2 -2\rho x_{k-1}y_{k-1})}{2(|y_{k-1}| + \varepsilon)}\right] \\
        t_2 \overset{\eqref{proposition:general_properties:5}}{=}&\mu\E\left[h_{k-1}g_{k-1}^2\right]\left[\frac{ \bar{f}(\cdot)(|y_{k-1}| + \varepsilon)}{2}\right].
    \end{align*}
    Thereafter, by using   \eqref{proposition:general_properties:2} we get that $\E\left[h_{k-1}\right] = 0$ and $\E\left[h_{k-1}g_{k-1}^2\right] = 0$, implying $t_1 = t_2 = 0$, thereby proving the claim. 
\end{proof}

\begin{lemma}
    \label{proposition:expectation_gamma_delta_k-1}
    Consider system \eqref{sys_def} with \Cref{assump:ghw}, and suppose \Cref{assump:quadratic} holds. Then, at any $k \geq 1$
    \begin{align}
        \label{proposition:equation:expectation_gamma_delta_k-1}
       \E\left[\mu h_{k-1}g_{k-1}\Delta_{k-1} \right]  & = \frac{\mu \gamma}{2} \Big(\rho^2\E\left[y_{k-1}^2\right] -2\rho\E\left[x_{k-1}y_{k-1}\right] \notag \\ &+ \psi\E\left[(|y_{k-1}| + \varepsilon)^2\right]\Big), 
    \end{align}
    where $\Delta_{k-1}$ is as defined in \eqref{proposition:DELTA}.
\end{lemma}
\begin{proof}[Proof of \Cref{proposition:expectation_gamma_delta_k-1}]
    For the proof we use the stochastic independence of the terms $x_{k-1},y_{k-1}$ from $w_{k-1}$. We start by expanding $\Delta_{k-1}$, thus getting
    \begin{align*}
        h_{k-1}g_{k-1}\Delta_{k-1} &= h_{k-1}g_{k-1}\Bigg((x_{k-1} - \rho y_{k-1})g_{k-1} \\ &\times(|y_{k-1}| + \varepsilon) + \frac{1}{2}\Big( \rho^2 y_{k-1}^2 \\ &+ g_{k-1}^2(|y_{k-1}| + \varepsilon)^2 - 2\rho x_{k-1}y_{k-1}\Big) \Bigg).
    \end{align*}
    By   \eqref{proposition:general_properties:5} we have $h_{k-1}g_{k-1}^2$ to be independent from $(x_{k-1} - \rho y_{k-1})(|y_{k-1}| + \varepsilon) $, and by   \eqref{proposition:general_properties:2} we have $\E\left[h_{k-1}g_{k-1}^2\right] = 0$, thereby $$\E\left[h_{k-1}g_{k-1}^2(x_{k-1} - \rho y_{k-1})(|y_{k-1}| + \varepsilon)\right] = 0.$$ For the remaining terms, we use the stochastic independence of $w_{k-1}$ and $x_{k-1},y_{k-1}$ as given in   \eqref{proposition:general_properties:5}, and further by   \eqref{proposition:general_properties:3} we get 
    \begin{align*}
        \frac{1}{2}\E\left[h_{k-1}g_{k-1}\Delta_{k-1}\right] &=\frac{1}{2}\E\Big[h_{k-1}g_{k-1}(\rho^2 y_{k-1}^2 \\ &- 2\rho x_{k-1}y_{k-1})\Big] \\ &+ \frac{1}{2}\E\big[h_{k-1}g_{k-1}^3(|y_{k-1}| + \varepsilon)^2\big] \\ &= \frac{\gamma \rho^2}{2} \E\left[y_{k-1}^2\right]  - \rho\gamma \E\left[x_{k-1}y_{k-1}\right] \\ &+ \frac{\gamma}{2} \E\left[g_{k-1}^2\right]\E\left[(|y_{k-1}| + \varepsilon)^2\right],
    \end{align*}
    where we observe that $\E\left[g_{k-1}^2\right] = \psi$, and by subsequent multiplication by $\mu$ we get the required expression.
\end{proof}

\noindent The following proposition is independent of the system analysis, and is necessary to find an upper bound for the expectation of an absolute value term in such a way that \Cref{lemma:matrix_upper_lower_bounds} is thereafter applicable.
\begin{proposition}
    \label{proposition:bounding_absolute}
    Consider any integrable random variable~$y$ with bounded mean and variance. Then, it applies that
    \begin{align}
        0 \leq \E\left[|y|\right] \leq \frac{1}{4} + \E\left[y^2\right].
    \end{align}
\end{proposition}
\begin{proof}[Proof of \Cref{proposition:bounding_absolute}]
    By the arithmetic and geometric mean inequality we have $y^2 + \frac{1}{4} \geq 2\sqrt{\frac{y^2}{4}} = |y| \geq 0$ for all $y \in \R$. Taking expectation, which is a linear operator, and by the boundedness of the terms, we get the required inequality.
\end{proof}

Next, we present the results of the analysis of the dynamics of system \eqref{sys_def} in expectation, and on the convergence of the system trajectories of \eqref{sys_def} in expectation to the global equilibrium $(x^{*}, 0)$.

\begin{lemma}
    \label{lemma:expectation_terms}
    Consider system \eqref{sys_def} with \Cref{assump:ghw}, and suppose \Cref{assump:quadratic} holds. Then at any $k\geq 1$
    \begin{subequations}
    \begin{align}
         \label{lemma:expectation_terms:1} \E\left[\Tilde{x}_{k+1}\right] &= \E\left[\Tilde{x}_{k}\right] -\rho\E\left[y_k\right], \\
        \label{lemma:expectation_terms:2}   \E\left[y_{k+1} \right] &= (1-\beta)\E\left[y_k\right] +\mu\gamma\E\left[\Tilde{x}_k\right].
    \end{align}
    \end{subequations}
\end{lemma}
\begin{proof}[Proof of \Cref{lemma:expectation_terms}]
      For the sake of simplicity we assume $x^{*} = 0$. The same results may be obtained by using $\Tilde{x}_k$ and $\Tilde{x}_{k+1}$ in~\eqref{sys_def}, which due to linearity does not influence the system dynamics. We start out with $x_{k+1}$, which we expand by \eqref{sys_def_x} and taking expectation, thus getting
      \begin{align*}
          \E\left[x_{k+1}\right] = \E\left[x_k\right] - \rho\E\left[y_k\right] + \E\left[(|y_k| + \varepsilon)g_k\right],
      \end{align*}
      where due to \Cref{proposition:general_properties} \ref{proposition:general_properties:5} and \ref{proposition:general_properties:2} we have the last term to be equal to zero. Hence,
      \begin{align*}
          \E\left[x_{k+1}\right] = \E\left[x_k\right] - \rho\E\left[y_k\right],
      \end{align*}
      and moreover 
      \begin{align}
        \label{proof:expectation:1}
          \E\left[x_{k}\right] = \E\left[x_{k-1}\right] - \rho\E\left[y_{k-1}\right].
      \end{align}
      Next, we use \eqref{sys_def_y}, plug in $\hat{h}(y_{k-1},w_{k-1}) = \frac{h_{k-1}}{|y_{k-1}| + \varepsilon}$ and take expectation to get
      \begin{align*}
          \E\left[y_{k+1}\right] = (1-\beta)\E\left[y_k\right] + \E\left[\frac{h_{k-1}\mu}{|y_{k-1}| + \varepsilon}\Delta_{k-1}\right],
      \end{align*}
      with $\Delta_{k-1}$ as in \eqref{proposition:DELTA}. Using \Cref{proposition:expectation_h_k-1_delta_k-1} with $\bar{f}(x_{k-1},y_{k-1},y_k) = 1$, we get
      \begin{align*}
          \E\left[y_{k+1}\right] = (1-\beta)\E\left[y_k\right] + \mu\gamma\left(\E\left[x_{k-1}\right] - \rho \E\left[y_{k-1}\right]\right),
      \end{align*}
      which by subsequently plugging in \eqref{proof:expectation:1} gives
      \begin{align*}
          \E\left[y_{k+1}\right] = (1-\beta)\E\left[y_k\right] + \mu\gamma\E\left[x_{k}\right].
      \end{align*}
\end{proof}
\begin{lemma}
    \label{lemma:expectation_convergence}
    Consider system \eqref{sys_def} with \Cref{assump:ghw}, suppose that \Cref{assump:quadratic} holds, and assume that the system parameters~${\chi, \psi, \rho \in \R^{+}}$ and $\beta \in (0,2)$ satisfy the inequality~${\mu\gamma\rho < \beta}$, where~${\gamma = \sqrt{\chi\psi}}$. Then for any $x_0,y_0 \in \R$ the system trajectory $(x_k,y_k)$ converges exponentially in expectation to $(x^{*}, 0)$.
\end{lemma}
\begin{proof}[Proof of \Cref{lemma:expectation_convergence}]
    The idea of the proof is to first represent the system dynamics in expectation in matrix form as an LTI system, and to use Jury's criterion to show stability, i.e. convergence by means of eigenvalue analysis. For the sake of simplicity assume $x^{*} = 0$, and start with \Cref{lemma:expectation_terms}, where we define
    \begin{align*}
        \nu_k := \begin{bmatrix}
            x_k & y_k
        \end{bmatrix}^{\top}.
    \end{align*}
    Observe that equations \eqref{lemma:expectation_terms:1} and \eqref{lemma:expectation_terms:2} may then be rewritten as
    \begin{align*}
        \E\left[\begin{matrix}
        x_{k+1} \\
        y_{k+1}
    \end{matrix}\right] &= \begin{bmatrix}
        1 & -\rho \\
        \mu\gamma & 1-\beta
    \end{bmatrix}\E\left[\begin{matrix}
        x_{k} \\
        y_{k}
    \end{matrix}\right],
    \end{align*}
    i.e. by using $\nu_k$ and $\nu_{k+1}$
    \begin{align*}
        \E[\nu_{k+1}] = \begin{bmatrix}
            1 & -\rho \\
            \mu\gamma & 1-\beta
        \end{bmatrix}\E[\nu_{k}] =: A_E \E[\nu_{k}].
    \end{align*}
    We now analyze the system as an LTI system, where to ensure stability we use Jury's criterion. We note here that the same results may be obtained by analyzing the sum and the product of the eigenvalues, i.e. if $\lambda_1, \lambda_2$ are the eigenvalues of $A_E$, then it must hold that
    \begin{align*}
        |\lambda_1 + \lambda_2| &< 2, & |\lambda_1\lambda_2| &< 1.
    \end{align*}
    In the following analysis we use Jury's criterion. The characteristic polynomial of $A_E$ is given by
    \begin{align*}
        p(\lambda) := \lambda^2 - (2-\beta)\lambda + 1-\beta + \mu\gamma\rho.
    \end{align*}
    For a polynomial of order $2$, the necessary conditions of Jury's criterion are sufficient as well. It is then necessary to have 
    \begin{align*}
        p(1) &> 0 & (-1)^2p(-1) &> 0, & 1 &> |1-\beta +\mu\gamma\rho|,
    \end{align*}
    i.e.
    \begin{alignat*}{2}
        & 1 -2+\beta + 1 -\beta + \mu\gamma\rho &&> 0 \\
        & 1 + 2 - \beta + 1-\beta + \mu\gamma\rho && > 0 \\
        &\text{if $1 - \beta + \mu\gamma\rho \geq 0$: } \hspace{0.5cm}\beta &&> \mu\gamma\rho \\
        &\text{if $1 - \beta + \mu\gamma\rho < 0$: }\mu\gamma\rho &&> -2 + \beta.
    \end{alignat*}
    The first inequality is equivalent to $\mu\gamma\rho > 0$ which is trivially true as $\mu,\gamma,\rho \in \R^{+}$. The second inequality leads to~${4 - 2\beta + \mu\gamma\rho > 0}$ which is true as $4-2\beta > 0$ for~${\beta \in (0,2)}$. This also leads to $-2+\beta < 0$ such that the last inequality is true. And finally, the 3rd inequality is true by assumption.
    Due to the property of linearity of this system, the results are valid for any $x_0,y_0 \in \R$. Moreover, by eigenvalue decomposition and coordinate system change it is easy to observe that the rate of convergence is governed by the magnitude of the highest eigenvalue of $A_E$. We thus conclude this system to be globally exponentially stable, i.e.~$\E\left[(x_k - x^{*})\right] \overset{k \to \infty}{\to} 0$ and $\E\left[y_k\right] \overset{k \to \infty}{\to} 0$.
\end{proof}

In the following, we present the results of the analysis of the dynamics of the second moment terms of system \eqref{sys_def} in expectation, for which we first provide the following definitions.  
We define the second moment vector for system~\eqref{sys_def} at time step $k \geq 1$ as
\begin{align}
    \label{def:zeta}
    \zeta_{k}&:= \begin{bmatrix}
        \Tilde{x}_{k-1}^2 &\Tilde{x}_{k}^2 &y_{k-1}^2 &y_{k}^2 &\Tilde{x}_{k-1}y_{k-1} &\Tilde{x}_{k}y_{k}
    \end{bmatrix}^{\top}.
\end{align}
We next define the square matrix $A_{ms}$ as follows
\begin{align}
    \label{A_ms_matrix}
    A_{ms} := &\left[\begin{smallmatrix}
        0 & 1 & 0 \\
        0 & 1 & 0 \\
        0 & 0 & 0 \\
        \mu^2\left(\gamma^2 + \rho^2\chi\right) & 0 &  \frac{\mu^2}{4}\left(\rho^4\chi + 6\gamma^2\rho^2 + \gamma^2\psi\right) \\
        0 & 0 & 0\\
        \mu\gamma & 0 & \frac{\mu\gamma}{2}\left(3\rho^2 + \psi\right)
    \end{smallmatrix}\color{white}\right]\color{black} \notag \\
    &\color{white}\left[\color{black}\begin{smallmatrix}
        0 & 0 & 0 \\
        \rho^2 + \psi & 0 & -2\rho \\
        1 & 0 & 0 \\
        (1-\beta)^2 & -\mu^2\rho\left(3\gamma^2 + \rho^2\chi\right) & 2(1-\beta)\mu\gamma \\
        0 & 0 & 1 \\
        -\rho(1-\beta) & -3\mu\rho\gamma & 1 - \beta -\mu\rho\gamma
    \end{smallmatrix}\right],
\end{align}
with positive real parameters $\mu, \chi, \psi, \rho$, and~${\beta \in (0,2)}$, where $\gamma = \sqrt{\chi\psi}$ as shown in \Cref{remark:gamma}.

\begin{lemma}
    \label{lemma:second_moment_terms}
    Consider system \eqref{sys_def} with \Cref{assump:ghw}, and suppose \Cref{assump:quadratic} holds. Then at any $k\geq 1$
    \begin{align}
         \label{lemma:second_moment_terms:1} \hspace{-0.2cm}\E\left[\Tilde{x}_{k+1}^2\right] &= \E\left[\Tilde{x}_{k}^2\right] + \left(\psi + \rho^2\right)\E\left[y_k^2\right] - 2\rho\E[\Tilde{x}_{k}y_k] \notag \\ &+ \varepsilon\psi (\varepsilon + 2\E\left[|y_k|\right]) \notag \\ &= A_{ms}^{(2,*)}\E\left[\zeta_k\right] + \varepsilon\psi(\varepsilon + 2\E\left[|y_k|\right]), \\
        \label{lemma:second_moment_terms:2}   \E\left[y_{k+1}^2 \right] &= (1-\beta)^2\E\left[y_k^2\right] + 2(1-\beta)\mu\gamma\E[\Tilde{x}_{k}y_k] \notag \\   &+ \mu^2\left(\rho^2\chi +  \gamma^2\right)\E\left[\Tilde{x}_{k-1}^2\right]   \notag \\ &- \mu^2\rho\left(3\gamma^2 + \rho^2\chi \right)\E\left[\Tilde{x}_{k-1}y_{k-1}\right] \notag \\ &+\frac{\mu^2\left(\rho^4\chi +\gamma^2\psi + 6\gamma^2\rho^2\right)}{4}\E\left[y_{k-1}^2\right]  + R_y(\varepsilon) \notag \\ &= A_{ms}^{(4,*)}\E\left[\zeta_k\right] + R_y(\varepsilon), 
    \end{align}
    where $\zeta_k$ and $A_{ms}$ are as in \eqref{def:zeta} and \eqref{A_ms_matrix}, respectively, and 
    \begin{align}
        \label{lemma:second_moment_terms:Ry}
        R_y(\varepsilon) &:= \mu^2\varepsilon\Bigg(\frac{\gamma^2\psi}{4}(\varepsilon + 2\E\left[|y_{k-1}|\right]) \notag \\ &-\frac{\rho^2\chi}{2}\E\left[T_1(y_{k-1},\varepsilon)\Big(\Tilde{x}_{k}^2 + \Tilde{x}_{k-1}^2\Big)\right] \notag \\ &+\frac{\rho^2(\rho^2\chi + 2\gamma^2)}{4}\E\left[T_1(y_{k-1},\varepsilon)y_{k-1}^2\right] \notag \\ &+\frac{\rho^2\gamma^2}{2}\E\left[T_2(y_{k-1},\varepsilon)\right] \Bigg), \\
        T_1(y,\varepsilon) &:= \frac{\varepsilon + 2|y|}{(|y| + \varepsilon)^2}, \\
        T_2(y,\varepsilon) &:= \frac{\varepsilon(\varepsilon + 2|y|)^2}{(|y| + \varepsilon)^2}.
    \end{align}
\end{lemma}
\begin{proof}[Proof of \Cref{lemma:second_moment_terms}]
    For the sake of simplicity we assume $x^{*} = 0$. The same results may be obtained by using $\Tilde{x}_k$ and $\Tilde{x}_{k+1}$ in~\eqref{sys_def}, which due to linearity does not influence the system dynamics. We start out with $\E\left[x_{k+1}^2\right]$, where we plug in~$\hat{g}(y_k,w_k) = (|y_k| + \varepsilon)g_k$, and by squaring \eqref{sys_def_x} and rearranging 
    \begin{align*}
        x_{k+1}^2 &= x_k^2 + (\rho^2 + g_k^2)y_k^2 - 2\rho x_ky_k \\ &+ 2(|y_k|+\varepsilon)g_k(x_k - \rho y_k) + \varepsilon(\varepsilon + 2|y_k|)g_k^2.
    \end{align*}
    From   \eqref{proposition:general_properties:5} with $$f(x_k,y_k,y_{k+1}) = (|y_k|+\varepsilon)(x_k - \rho y_k)$$ and further from   
    \eqref{proposition:general_properties:2} we get 
    \begin{align*}
        \E\left[(|y_k|+\varepsilon)g_k(x_k - \rho y_k) \right] &=\E\left[(|y_k|+\varepsilon)(x_k - \rho y_k) \right] \\ &\times \E\left[g_k\right] = 0,
    \end{align*}
    thus by taking expectation and employing   \eqref{proposition:general_properties:5} we get
    \begin{align*}
        \E\left[x_{k+1}^2\right] &= \E\left[x_k^2\right] + (\rho^2 + \psi)\E\left[y_k^2\right] - 2\rho \E\left[x_ky_k\right] \\ &+ \varepsilon\psi(\varepsilon + 2\E\left[|y_k|\right]).
    \end{align*}
    Next, by multiplying out $A_{ms}^{(2,*)}\E\left[\zeta_k\right]$ we observe that $\E\left[x_{k+1}^2\right] = A_{ms}^{(2,*)}\E\left[\zeta_k\right] + \varepsilon\psi(\varepsilon + 2\E\left[|y_k|\right])$, thus we prove relation \eqref{lemma:second_moment_terms:1}.

    We proceed with \eqref{lemma:second_moment_terms:2}. Using \eqref{sys_def_y}, \Cref{proposition:quadratic_decomposition} and plugging in $\hat{h}(y_{k-1},w_{k-1}) = \frac{h_{k-1}}{|y_{k-1}| + \varepsilon}$, we rearrange the terms in $y_{k+1}^2$ to get the equivalent relation
    \begin{align*}
        y_{k+1}^2 &= (1-\beta)^2y_k^2 + \mu^2\frac{h_{k-1}^2}{(|y_{k-1}| + \varepsilon)^2}\Delta_{k-1}^2 \\ &+ 2\mu{(1-\beta)}\frac{ h_{k-1}y_k}{|y_{k-1}| + \varepsilon}\Delta_{k-1},
    \end{align*}
    with $\Delta_{k-1}$ as in \eqref{proposition:DELTA}. By using \Cref{proposition:expectation_h_k-1_delta_k-1} with $\bar{f}(x_{k-1},y_{k-1},y_k) = y_k$ we obtain 
    \begin{align*}
        2\mu{(1-\beta)}\E\left[\frac{ h_{k-1}y_k}{|y_{k-1}| + \varepsilon}\Delta_{k-1}\right] &= 2(1-\beta)\mu  \gamma \Big(\E[x_{k-1}y_k] \\ &-\rho \E[y_{k-1}y_k]\Big),
    \end{align*}
    where we use \eqref{proposition:general_properties:5} with $f(x_{k-1},y_{k-1}y_k) = (|y_{k-1}|+\varepsilon)y_k$ at time step $k-1$ to claim stochastic independence from the term $g_{k-1}$, thereby
    \begin{align}
        \label{proof:second_moment_lemma:1}
        \E[(|y_{k-1}|+\varepsilon)g_{k-1}y_k]  &= \E[(|y_{k-1}|+\varepsilon)y_k]\E\left[g_{k-1}\right]  \overset{\eqref{proposition:general_properties:2}}{=} 0.
    \end{align}
    Then, as from \eqref{sys_def_x} we have $$x_{k} = x_{k-1} - \rho y_{k-1} + (|y_{k-1}|+\varepsilon)g_{k-1},$$ by multiplying with $y_k$, taking expectation and plugging in the zero mean term \eqref{proof:second_moment_lemma:1} we get
    \begin{align}
    \label{proof:lemma:second_moments_2}
        2\mu{(1-\beta)}\E\left[\frac{ h_{k-1}y_k}{|y_{k-1}| + \varepsilon}\Delta_{k-1}\right] &= 2(1-\beta)\mu  \gamma\E[x_{k}y_k].
    \end{align}  
    Next, by plugging in \eqref{proposition:DELTA} in $\mu^2\frac{h_{k-1}^2}{(|y_{k-1}| + \varepsilon)^2}\Delta_{k-1}^2$ and rearranging we get
    \begin{align*}
        h_{k-1}^2\mu^2\Bigg[&(x_{k-1}-\rho y_{k-1})^2g_{k-1}^2 + \frac{g_{k-1}^4}{4}(|y_{k-1}| + \varepsilon)^2 \\
         +& \frac{\rho^2g_{k-1}^2}{2}y_{k-1}^2 - \rho x_{k-1}y_{k-1}g_{k-1}^2
        \\ +& \frac{\rho^2 y_{k-1}^2}{(|y_{k-1}| +  \varepsilon)^2}\left(x_{k-1}^2 + \frac{\rho^2}{4}y_{k-1}^2 - \rho x_{k-1}y_{k-1} \right) \\ +& 2\frac{(x_{k-1} - \rho y_{k-1})g_{k-1}}{|y_{k-1}| + \varepsilon}\Big(-\rho x_{k-1}y_{k-1} + \frac{\rho^2}{2}y_{k-1}^2  \\ +& \frac{g_{k-1}^2}{2}(|y_{k-1}| + \varepsilon)^2\Big)  \Bigg].
    \end{align*}
    For the term $\frac{\rho^2 y_{k-1}^2}{(|y_{k-1}| +  \varepsilon)^2}$ we now use $y_{k-1}^2 = (|y_{k-1}| +\varepsilon - \varepsilon)^2 $, which is further equal to $(|y_{k-1}| + \varepsilon)^2 - 2\varepsilon|y_{k-1}| - \varepsilon^2$. Thus, by rearrangement we get
    \begin{align*}
        h_{k-1}^2\mu^2\Bigg[&x_{k-1}^2(\rho^2 + g_{k-1}^2)  - \rho (3g_{k-1}^2 + \rho^2) x_{k-1}y_{k-1}  \\ +& y_{k-1}^2 \left(\frac{6\rho^2 g_{k-1}^2 + \rho^4 + g_{k-1}^4}{4}\right)  \\+& \frac{g_{k-1}^4}{4}(2|y_{k-1}| + \varepsilon) - \rho^2\frac{\varepsilon (2|y_{k-1}| + \varepsilon)}{(|y_{k-1}| +  \varepsilon)^2} \\ \times&\left(x_{k-1}^2 + \frac{\rho^2}{4}y_{k-1}^2 -\rho x_{k-1}y_{k-1} \right) \\ +& 2\frac{(x_{k-1} - \rho y_{k-1})g_{k-1}}{|y_{k-1}| + \varepsilon}\Big(-\rho x_{k-1}y_{k-1} + \frac{\rho^2}{2}y_{k-1}^2  \\ +& \frac{g_{k-1}^2}{2}(|y_{k-1}| + \varepsilon)^2\Big)  \Bigg].
    \end{align*}
    We rewrite the following term as
    \begin{align*}
        &2h_{k-1}^2\mu^2\frac{(x_{k-1} - \rho y_{k-1})g_{k-1}}{|y_{k-1}| + \varepsilon}\Big(-\rho x_{k-1}y_{k-1} + \frac{\rho^2}{2}y_{k-1}^2  \\ &+ \frac{g_{k-1}^2}{2}(|y_{k-1}| + \varepsilon)^2\Big)  =: \tilde{C}\tilde{f}(x_{k-1},y_{k-1})h_{k-1}^2g_{k-1}
    \end{align*}
   where $\tilde{C}$ is a constant term, and due to  \eqref{proposition:general_properties:5} and \eqref{proposition:general_properties:2} the above term has expectation equal to $0$. Further, expanding $x_k$ w.r.t. time step $k-1$ using \eqref{sys_def_x}, squaring and rearranging we get
   \begin{align*}
       -\rho x_{k-1}y_{k-1} &= \frac{1}{2}x_{k}^2 - \frac{1}{2}x_{k-1}^2 - \frac{\rho^2+g_{k-1}^2}{2}y_{k-1}^2 \\ &-g_{k-1}(x_{k-1} - \rho y_{k-1})(|y_{k-1}| + \varepsilon) \\ &- \frac{\varepsilon}{2}(\varepsilon + 2|y_{k-1}|)g_{k-1}^2,
   \end{align*}
   which we then plug into $\frac{\varepsilon\rho^3(2|y_{k-1}| + \varepsilon)}{(|y_{k-1}| + \varepsilon)^2}x_{k-1}y_{k-1}$.
   By taking the expectation of $\mu^2\frac{h_{k-1}^2}{(|y_{k-1}|+\varepsilon)^2}\Delta_{k-1}^2$, employing   \eqref{proposition:general_properties:5} and~\eqref{proposition:general_properties:3}, and rearranging, we get
   \begin{align}
    \label{proof:lemma:second_moments_3}
        \E\left[\frac{\mu^2h_{k-1}^2\Delta_{k-1}^2}{(|y_{k-1}| + \varepsilon)^2}\right]=&\mu^2(\rho^2\chi + \gamma^2)\E\left[x_{k-1}^2\right] \notag \\+&  \mu^2\frac{\gamma^2\psi + 6\rho^2\gamma^2 + \rho^4\chi}{4}\E\left[y_{k-1}^2\right] \notag\\ -& \mu^2\rho(3\gamma^2 + \rho^2\chi)\E\left[x_{k-1}y_{k-1}\right] \notag\\ +& \frac{\mu^2\gamma^2\psi\varepsilon}{4}(\varepsilon + 2\E\left[|y_{k-1}|\right]) \notag\\ +&\frac{\mu^2\rho^2\gamma^2}{2}\E\left[\frac{\varepsilon^2(\varepsilon+2|y_{k-1}|)^2}{(|y_{k-1}| + \varepsilon)^2}\right] \notag\\ -&\frac{\mu^2\rho^2\chi}{2}\E\left[\varepsilon T_1(y_{k-1},\varepsilon)(x_{k-1}^2 + x_{k}^2)\right] \notag \\+&
        \frac{\mu^2\rho^2(2\gamma^2 + \rho^2\chi)}{4} \notag \\  \times&\E\left[\varepsilon T_1(y_{k-1},\varepsilon)y_{k-1}^2\right],
    \end{align}
    where $T_1(y_{k-1},\varepsilon) = \frac{(2|y_{k-1}| + \varepsilon)}{(|y_{k-1}| + \varepsilon)^2}$ as defined in \Cref{lemma:second_moment_terms}.
   By summing up $(1-\beta)^2\E\left[y_k^2\right]$ with \eqref{proof:lemma:second_moments_2} and \eqref{proof:lemma:second_moments_3}, we get the required expression. To finally prove that $\E\left[y_{k+1}^2\right] = A_{ms}^{(4,*)}\E\left[\zeta_k\right] + R_y(\varepsilon)$, one simply needs to multiply $A_{ms}^{(4,*)}\E\left[\zeta_k\right]$ and compare with the the sum of $(1-\beta)^2\E\left[y_k^2\right]$, \eqref{proof:lemma:second_moments_2} and \eqref{proof:lemma:second_moments_3}.
\end{proof}

\begin{lemma}
    \label{lemma:second_moment_aux_terms}
    Consider system \eqref{sys_def} with \Cref{assump:ghw}, and suppose \Cref{assump:quadratic} holds. Then at any $k\geq 1$
    \begin{align}
        \label{lemma:equation:second_moment_aux_terms}
        \E[\Tilde{x}_{k+1}y_{k+1}] &= -\rho(1-\beta)\E[y_k^2] + \mu\gamma\E\left[\Tilde{x}_{k-1}^2\right]  \notag \\ &+\mu\gamma\frac{3\rho^2 + \psi}{2}\E\left[y_{k-1}^2\right]  \notag \\ & -3\mu\rho\gamma\E\left[\Tilde{x}_{k-1}y_{k-1}\right]  \notag \\ & + (1-\beta - \mu\rho\gamma)\E\left[\Tilde{x}_{k}y_k\right] \notag \\ &+ \frac{\mu\gamma\psi\varepsilon}{2}(\varepsilon + 2\E\left[|y_{k-1}|\right]) \notag \\ &= A_{ms}^{(6,*)}\E\left[\zeta_k\right] + \frac{\mu\gamma\psi\varepsilon}{2}(\varepsilon + 2\E\left[|y_{k-1}|\right]), 
    \end{align}
    where $\zeta_k$ and $A_{ms}$ are as in \eqref{def:zeta} and \eqref{A_ms_matrix}, respectively.
\end{lemma}
\begin{proof}[Proof of \Cref{lemma:second_moment_aux_terms}]
    For the sake of simplicity we assume ${x^{*} = 0}$. The same results may be obtained by using $\Tilde{x}_k$ and $\Tilde{x}_{k+1}$ in \eqref{sys_def}, which due to linearity does not influence the system dynamics. 
    We shall use the properties stated in \Cref{proposition:general_properties}, and start by plugging in $\hat{g}(y_k,w_k) = (|y_k| + \varepsilon)g_k$ and expanding $x_{k+1}$ using \eqref{sys_def_x}, thus getting
    \begin{align*}
        \E[x_{k+1}y_{k+1}] &= \E[x_ky_{k+1}] - \rho\E[y_ky_{k+1}] \\ &+ \E[(|y_k|+\varepsilon)g_ky_{k+1}].
    \end{align*}
    For the last term   \eqref{proposition:general_properties:5} applies with $$f(x_k,y_k,y_{k+1}) = (|y_k| + \varepsilon)y_{k+1},$$ and further due to  \eqref{proposition:general_properties:2} we get that
    \begin{align}
        \label{proof:lemma:second_aux_moments:star}
        \E[(|y_k|+\varepsilon)g_ky_{k+1}] = \E[(|y_k|+\varepsilon)y_{k+1}]\E\left[g_{k}\right] =0,
    \end{align}
    and we reuse this property throughout the proof.
    We proceed with $\E\left[y_ky_{k+1}\right]$. Using \eqref{sys_def_y} to expand $y_{k+1}$ using $\Delta_{k-1}$ from \Cref{proposition:quadratic_decomposition} and $\hat{h}(y_{k-1},w_{k-1}) = \frac{h_{k-1}}{|y_{k-1}|+\varepsilon}$ we get 
    \begin{align*}
        \E[y_ky_{k+1}] &= {(1-\beta)\E[y_k^2]} + \mu\E\left[\frac{y_kh_{k-1}}{|y_{k-1}|+\varepsilon}\Delta_{k-1}\right].
    \end{align*}
    By using \Cref{proposition:expectation_h_k-1_delta_k-1} with $\bar{f}(x_{k-1},y_{k-1},y_k) = y_k$, we get 
    \begin{align*}
        \E[y_ky_{k+1}] &= (1-\beta)\E[y_k^2] + \mu\gamma\E\left[x_{k-1}y_k\right] \notag \\ &- \mu\rho \gamma\E\left[y_{k-1}y_k\right].
    \end{align*}
    As previously shown, from \eqref{proposition:general_properties:5} and \eqref{proposition:general_properties:2} it follows that~$\E[(|y_{k-1}| + \varepsilon)g_{k-1}y_{k}] = 0$. Thus,
    \begin{align*}
        \E\left[x_ky_k\right]&\overset{\eqref{sys_def}}{=} 
        \E\left[(x_{k-1}-\rho y_{k-1} + (|y_{k-1}| + \varepsilon)g_{k-1})y_k\right]\\
        &=\E\left[x_{k-1}y_k\right] - \rho\E\left[y_{k-1}y_k\right],
    \end{align*}
    and thereby,
    \begin{align}
        \label{proof:lemma:second_moment_aux_terms:0}
        \E[y_ky_{k+1}] &= (1-\beta)\E[y_k^2] + \mu\gamma\E\left[x_{k}y_k\right].
    \end{align}
    We proceed with $\E[x_ky_{k+1}]$. Using \eqref{sys_def_x} to expand $x_k$, we get 
    \begin{align*}
        \E[x_ky_{k+1}] &= \E\left[x_{k-1}y_{k+1}\right] - \rho \E\left[y_{k-1}y_{k+1}\right] \\ &+ \E\left[(|y_{k-1}|+\varepsilon)g_{k-1}y_{k+1}\right]
    \end{align*}
    By expanding $y_{k+1}$ using \eqref{sys_def_y}, for the first term we get 
    \begin{align*}
        \E\left[x_{k-1}y_{k+1}\right] &= (1-\beta)\E[x_{k-1}y_k] \\ &+ \mu\E\left[\frac{x_{k-1}h_{k-1}}{|y_{k-1}|+\varepsilon}\Delta_{k-1}\right]
    \end{align*}
    By further taking into account \Cref{proposition:expectation_h_k-1_delta_k-1} with~${\bar{f}(x_{k-1},y_{k-1},y_k) = x_{k-1}}$ we get 
    \begin{align}
        \label{proof:lemma:second_moment_aux_terms:1}
        \E\left[x_{k-1}y_{k+1}\right] &= (1-\beta)\E[x_{k-1}y_k] + \mu\gamma\E\left[x_{k-1}^2\right] \notag \\ &- \mu\rho\gamma\E[x_{k-1}y_{k-1}].
    \end{align}
    Next, by expanding $y_{k+1}$ using \eqref{sys_def_y} we get 
    \begin{align*}
        \E\left[y_{k-1}y_{k+1}\right] &= (1-\beta)\E\left[y_{k-1}y_k\right] \\ &+\mu\E\left[\frac{y_{k-1}h_{k-1}}{|y_{k-1}| + \varepsilon}\Delta_{k-1}\right]
    \end{align*}
    where we once again refer to \Cref{proposition:expectation_h_k-1_delta_k-1} with~${\bar{f}(x_{k-1},y_{k-1},y_k) = y_{k-1}}$ to get
    \begin{align}
        \label{proof:lemma:second_moment_aux_terms:2}
        \E\left[y_{k-1}y_{k+1}\right] &= (1-\beta)\E[y_{k-1}y_k]\notag \\ & + \mu\gamma\left(\E\left[x_{k-1}y_{k-1}\right]  - \rho\E\left[y_{k-1}^2\right]\right).
    \end{align}
    By expanding $y_{k+1}$ using \eqref{sys_def_y}, we further get
    \begin{align*}
        \E[(|y_{k-1}|+\varepsilon)g_{k-1}y_{k+1}] &= (1-\beta)\E[(|y_{k-1}| + \varepsilon)g_{k-1}y_k] \\ &+ \mu\E\left[h_{k-1}g_{k-1}\Delta_{k-1}\right],
    \end{align*}
    where we observe that
    \begin{align*}
        \E[(|y_{k-1}| + \varepsilon)g_{k-1}y_k] &\overset{\eqref{proposition:general_properties:5}}{=} \E[(|y_{k-1}| + \varepsilon)y_k]\E\left[g_{k-1}\right] \\ &\overset{\eqref{proposition:general_properties:2}}{=} 0.
    \end{align*}
    Next, we have 
    \begin{align}
        \label{proof:lemma:second_moment_aux_terms:3}
        \E\left[(|y_{k-1}| + \varepsilon)g_{k-1}y_{k+1}\right] &\overset{\eqref{sys_def_y}}{=} 
        (1-\beta) \notag \\ & \times \E\left[(|y_{k-1}| + \varepsilon)g_{k-1}y_{k}\right] \notag \\ &+ \mu\E\Big[(|y_{k-1}| + \varepsilon)\frac{h_{k-1}g_{k-1}}{|y_{k-1}| + \varepsilon} \notag \\ & \times \Delta_{k-1}\Big] \notag \\ &\overset{ \eqref{proposition:general_properties:2}}{=} \mu\E\Big[h_{k-1}g_{k-1}\Delta_{k-1}\Big]
        \notag \\ &\overset{\eqref{proposition:equation:expectation_gamma_delta_k-1}}{=} \mu\gamma\frac{\rho^2 + \psi}{2}\E\left[y_{k-1}^2\right]  \notag \\ &-\mu\rho\gamma\E\left[x_{k-1}y_{k-1}\right] \notag \\ &+\mu\gamma\frac{\varepsilon\psi}{2}(\varepsilon + 2\E\left[|y_{k-1}|\right]).
    \end{align}
    Thus, by first scaling \eqref{proof:lemma:second_moment_aux_terms:2} by $-\rho$ and by summing it with \eqref{proof:lemma:second_moment_aux_terms:1} and \eqref{proof:lemma:second_moment_aux_terms:3}, we get
    \begin{align}
    \label{proof:lemma:second_moment_aux_terms:5}
        \E[x_ky_{k+1}] &= \mu\gamma\frac{3\rho^2 + \psi}{2}\E\left[y_{k-1}^2\right]  \notag \\ & -3\mu\rho\gamma\E\left[x_{k-1}y_{k-1}\right] + \mu\gamma\E\left[x_{k-1}^2\right] \notag \\ & + (1-\beta)\E[(x_{k-1}-\rho y_{k-1})y_k]\notag \\
        &+\mu\gamma\frac{\varepsilon\psi}{2}(\varepsilon + 2\E\left[|y_{k-1}|\right]) \notag \\
        &\overset{\eqref{proof:lemma:second_aux_moments:star} + \eqref{sys_def_x}}{=}\mu\gamma\frac{3\rho^2 + \psi}{2}\E\left[y_{k-1}^2\right]  \notag \\ 
        & -3\mu\rho\gamma\E\left[x_{k-1}y_{k-1}\right] + \mu\gamma\E\left[x_{k-1}^2\right] \notag \\ & + (1-\beta)\E[x_{k}y_k] \notag \\ &+\mu\gamma\frac{\varepsilon\psi}{2}(\varepsilon + 2\E\left[|y_{k-1}|\right]),
    \end{align}
    as $(1-\beta)\E[(|y_{k-1}| + \varepsilon)g_{k-1}y_k] = 0$.
    By scaling \eqref{proof:lemma:second_moment_aux_terms:0} by $-\rho$, and summing it up with \eqref{proof:lemma:second_moment_aux_terms:5} we get the required expression, which we then compare to $$A_{ms}^{(6,*)}\E\left[\zeta_k\right] + \mu\gamma\frac{\varepsilon\psi}{2}(\varepsilon + 2\E\left[|y_{k-1}|\right]),$$ thereby concluding the proof.
\end{proof}

In \Cref{lemma:second_moment_terms} and \Cref{lemma:second_moment_aux_terms}, we have shown that the expectation of the second moment terms at the next time step is linearly dependent on the second moments of the current and previous time steps, as well as on some disturbance term, i.e. terms containing the multiplicand $\varepsilon$. The following lemma is necessary for the analysis of the dynamics of system \eqref{sys_def} in the expectation of the second moment for the limit case of $\varepsilon \to 0$, and for the choice of system parameters.
\begin{lemma}
    \label{lemma:stable_matrices}
    Consider the matrix $A_{ms}$ as defined in \eqref{A_ms_matrix}.
    Then, for every $\chi, \psi$ and~$\mu$ there exist some~${\rho_0 \in \R^{+}}$,~${\beta_0 \in (0,1)}$ such that for any $\rho \in (0, \rho_0]$ and ${\beta \in [1-\beta_0, 1+\beta_0]}$ all eigenvalues of $A_{ms}$ are inside the unit circle.
\end{lemma}
\begin{proof}[Proof of \Cref{lemma:stable_matrices}]
    The idea of the proof is to analyze the eigenvalues of $A_{ms}$ through the use of the Jury criterion. For the proof we choose $\chi := \frac{q_1}{\mu^2}$ and $\psi := q_2\rho^4$, where $q_1,q_2 \in \R^{+}$ are arbitrary weights. In such a way we argue that $\chi$ and $\psi$ may take any arbitrary positive value, hence this analysis is not restrictive. As $\gamma = \sqrt{\chi\psi}$, i.e. $\gamma^2 = \chi\psi$, observe then that 
    \begin{align*}
        \mu^2\chi &= \mu^2 \frac{q_1}{\mu^2} = q_1, \\
        \mu\gamma &= \mu\sqrt{\frac{q_1}{\mu^2}q_2\rho^4} = \sqrt{q_1q_2}\rho^2, \\
        -\mu^2\rho(3\gamma^2 + \rho^2\chi) &= - \mu^2\rho\chi(3\psi + \rho^2) \\ &= -q_1(3q_2\rho^5 + \rho^3),
    \end{align*}
    from where it is easy to observe that $A_{ms}$ is independent of $\mu$. Thus, the following analysis is based on the parameters $\rho$ and $\beta$, in addition to $q_1, q_2$. We now proceed to analyze the eigenvalues of $A_{ms}$, and obtain the equation
    \begin{align*}
        0 = \lambda(\lambda^5 + a_4\lambda^4 + a_3\lambda^3 + a_2\lambda^2 + a_1\lambda + a_0) =: \lambda p(\lambda),
    \end{align*}
    where
    \begin{align*}
        a_4 &:= 3\beta - \beta^2 + (q_1q_2)^{\frac{1}{2}}\rho^3 - 3\\
        a_3 &:= 3 -6\beta - 0.25q_1\rho^4 + 4\beta^2 - \beta^3 - 1.5q_1q_2\rho^6 \\ &- 0.25q_1q_2^2\rho^8 + (3 - 2\beta + \beta^2)(q_1q_2)^{\frac{1}{2}}\rho^3 \\
        a_2 &:= 3\beta - 1.5q_1\rho^4 - 3\beta^2 + \beta^3 + 0.75\beta q_1\rho^4 - 5q_1q_2\rho^6 \\ &- 1.5q_1q_2^2\rho^8 - 5(q_1q_2)^{\frac{1}{2}}\rho^3 - 0.25q_1^{\frac{3}{2}}q_2^{\frac{1}{2}}\rho^7 \\ &- 1.5q_1^{\frac{3}{2}}q_2^{\frac{3}{2}}\rho^9 - 0.25q_1^{\frac{3}{2}}q_2^{\frac{5}{2}}\rho^{11} + 8\beta (q_1q_2)^{\frac{1}{2}}\rho^3 \\ &+ 4.5\beta q_1q_2\rho^6 - 4\beta^2(q_1q_2)^{\frac{1}{2}}\rho^3 + 0.75\beta q_1q_2^2\rho^8 - 1 \\
        a_1 &:= - 0.25q_1\rho^4 + 0.25\beta q_1\rho^4 + 2.5q_1q_2\rho^6) - 0.25q_1q_2^2\rho^8 \\ &+ (q_1q_2)^{\frac{1}{2}}\rho^3 - 2\beta (q_1q_2)^{\frac{1}{2}}\rho^3 - 2.5\beta q_1q_2\rho^6 \\ &+ \beta^2(q_1q_2)^{\frac{1}{2}}\rho^3 + 0.25\beta q_1q_2^2\rho^8 \\
        a_0 &:= - 0.25q_1^{\frac{3}{2}}q_2^{\frac{1}{2}}\rho^7 - 1.5(q_1q_2)^{\frac{3}{2}}\rho^9 - 0.25q_1^{\frac{3}{2}}q_2^{\frac{5}{2}}\rho^{11}.
    \end{align*}
    Due to Abel's Impossibility Theorem, there exist no closed form solutions for $p(\lambda) = 0$. 
    We resort to using the Jury criterion for the $5$-th order polynomial. Substitute $\beta = 1 + \epsilon$, where $\epsilon \in (-1,1)$, and denote by $\mathcal{O}(\epsilon^p,\rho^q)$ any term which has at least one of the multiplicand terms $\epsilon^p,\rho^q$, where~${ p,q \in \N}$. The necessary conditions of Jury's criterion are $p(1) >0$, $(-1)^{5}p(-1) > 0$ and the coefficient of the highest order term to be positive, where we note the latter to be satisfied. We thus obtain
    \begin{align*}
        p(1) &= - q_1\rho^4(1-\epsilon) + 2(q_1q_2)^{\frac{1}{2}}\rho^3(1- \epsilon^2) \\ &+ \mathcal{O}(\epsilon^3, \rho^5), \\
        -p(-1) &= 2 - 2\epsilon + 2\epsilon^2 + 0.5q_1\rho^4(1- \epsilon)  \\ &+ (2 + 6\epsilon^2)(q_1q_2)^{\frac{1}{2}}\rho^3 + 
        \mathcal{O}(\epsilon^3, \rho^5),
    \end{align*}
    from where it follows that given some $q_1, q_2 > 0$ and a sufficiently small $\rho> 0$ we have $p(1) > 0$, and in addition if~$\epsilon$ is sufficiently small we get $-p(-1) > 0$, hence the necessary conditions of Jury's criterion are fulfilled. The sufficient conditions of Jury's criterion are obtained by the determinant method, and are calculated for order $5$ down to order $2$. Denote by $a_{i,j}$ the coefficient of polynomial order $2 \leq i \leq 5$ of term $0 \leq j \leq i$, as for example in
    \begin{align*}
        p_4(\lambda) = a_{4,4}\lambda^4 + a_{4,3}\lambda^3 + a_{4,2}\lambda^2 + a_{4,1}\lambda + a_{4,0}. 
    \end{align*}
    The sufficient conditions are as shown in continuation: for~${i = 5}$, it must hold that ${a_{5,5} > |a_{5,0}|}$, and for $i = 4,3,2$, it must hold that $|a_{i,i}| < |a_{i,0}|$. In exact, we obtain the inequalities
    \begin{align*}
        1 &> | a_0| \\
        |(-1) + \mathcal{O}(\epsilon, \rho^6)| & > |\mathcal{O}(\epsilon, \rho^3)| \\
        |1 + \mathcal{O}(\epsilon, \rho^3)| & > | \mathcal{O}(\epsilon^2, \rho^3) | \\
        |1 + \mathcal{O}(\epsilon, \rho^3)| & > | -\epsilon + \mathcal{O}(\epsilon, \rho^3)|,
    \end{align*}
    for $i = 5,4,3,2$ respectively. For all inequalities, we observe that as $\rho \to 0$ and $ \epsilon \to 0$, the right hand side approaches $0$ and the left hand side approaches $1$, thus the left hand side constant dominates the right hand side terms for some sufficiently small parameters $\epsilon,\rho$, and this interval may be influenced by the choice of $q_1$ and $q_2$. We thus conclude the sufficient conditions to be satisfied, and thereby conclude that the matrix eigenvalues are inside the unit circle.
\end{proof}

\subsection{Proof of Main Result}
\label{section_IV}
\begin{proof}[Proof of \Cref{theorem:convergence_means_square}]
    Assume $x^{*} = 0$ for ease of notation, hence $\tilde{x}_k = x_k$ for any $k \geq 0$. We aim to use \Cref{lemma:matrix_upper_lower_bounds} to show global exponential practical convergence of the second moment to the equilibrium, and moreover by using \Cref{lemma:expectation_convergence} to claim global exponential convergence in expectation of the system iterates. Throughout the proof, we use that~${\varepsilon \in (0, \varepsilon_0]}$, and for convenience we denote $\varepsilon \in [\epsilon_0, \varepsilon_0]$, where $\varepsilon_0 > \epsilon_0 > 0$.
    
    In the following we aim to find upper and lower bounds for the dynamics given in \Cref{lemma:second_moment_terms} and \Cref{lemma:second_moment_aux_terms}, such that we may use \Cref{lemma:matrix_upper_lower_bounds}. 
    We start by considering \Cref{proposition:bounding_absolute}, from where we have
    \begin{align}   
        \label{proof:theorem:ineq_2}
        0 \leq \E\left[|y_{k-1}|\right] \leq  \frac{1}{4} + \E\left[y_{k-1}^2\right].
    \end{align}
    We use the aforementioned inequalities to bound the term~$\E\left[x_{k+1}^2\right]$, whose dynamics are given in \eqref{lemma:second_moment_terms:1}. By using~\eqref{proof:theorem:ineq_2}, $\psi\varepsilon^2 \geq \psi\varepsilon\epsilon_0$ and $\psi\varepsilon^2 \leq \psi \varepsilon\varepsilon_0$ we get
    \begin{subequations}
        \begin{align}
            \label{proof:theorem_1:inequality_x^2}
            \E\left[x_{k+1}^2\right] &\geq  A_{ms}^{(2,*)}\E\left[\zeta_k\right] + \psi\varepsilon^2 \geq  A_{ms}^{(2,*)}\E\left[\zeta_k\right] + \psi\varepsilon\epsilon_0\notag \\ &=: A_{ms}^{(2,*)}\E\left[\zeta_k\right] + \varepsilon Q_1^{(2,*)}\E\left[\zeta_k\right] + \varepsilon b_1^{(2)} \\
            \label{proof:theorem_1:inequality_x^2.2}
            \E\left[x_{k+1}^2\right] &\leq  A_{ms}^{(2,*)}\E\left[\zeta_k\right] + \psi\varepsilon^2 + \frac{1}{2}\psi\varepsilon + 2\psi\varepsilon\E\left[y_k^2\right] \notag \\
            &\leq  A_{ms}^{(2,*)}\E\left[\zeta_k\right] + \psi\varepsilon\varepsilon_0 + \frac{1}{2}\psi\varepsilon + 2\psi\varepsilon\E\left[y_k^2\right] \notag \\
            &=: A_{ms}^{(2,*)}\E\left[\zeta_k\right] + \varepsilon Q_2^{(2,*)}\E\left[\zeta_k\right] + \varepsilon b_2^{(2)},
        \end{align}
    \end{subequations}
    with 
    \begin{subequations}
        \label{proof:main_theorem:implicit_definitions}
        \begin{align}
            Q_1^{(2,*)} &:= \begin{bmatrix}     
                                0 & 0 & 0 & 0 & 0 & 0
                            \end{bmatrix}, \\
            Q_2^{(2,*)} &:= \begin{bmatrix}     
                                0 & 0 & 0 & 2\psi & 0 & 0
                            \end{bmatrix}, \\
            b_1^{(2)} &:= \psi\varepsilon\epsilon_0, \\
            b_2^{(2)} &:= \psi\left(\frac{1}{2} + \varepsilon_0\right).
        \end{align}
    \end{subequations}
    
    We continue by considering the terms $T_1(y,\varepsilon)=\frac{\varepsilon + 2|y|}{(|y| + \varepsilon)^2}$ and~${T_2(y,\varepsilon):= \frac{\varepsilon(\varepsilon + 2|y|)^2}{(|y| + \varepsilon)^2}}$ as defined in \Cref{lemma:second_moment_terms}. We highlight that both $T_1(y,\varepsilon)$ and $T_2(y,\varepsilon)$ are  non-negative and as $|y| + \varepsilon > 0$ for all $ y  \in \R$, it is easy to prove they are upper bounded. Thus, there exist some positive constants $${M > 0}, F>0: {0 \leq T_1(y,\varepsilon) \leq M}, {0 \leq T_2(y,\varepsilon) \leq F}$$ for all $ y \in \R$ and $\varepsilon \in (0, \varepsilon_0]$.
    For bounding the term $R_y(\varepsilon)$ given in \eqref{lemma:second_moment_terms:Ry} we first bound the following expectations 
    \begin{subequations}
    \label{proof:theorem:ineq_1}
        \begin{alignat}{2}
            \label{proof:theorem:ineq_1_1}
            -M\E\left[x_{k-1}^2\right] &\leq -\E\left[ T_1(y_{k-1,\varepsilon})x_{k-1}^2\right] &&\leq 0, \\
            - M\E\left[x_{k}^2\right] &\leq -\E\left[ T_1(y_{k-1},\varepsilon)x_{k}^2\right] &&\leq 0, \\
            0 &\leq \E\left[T_1(y_{k-1},\varepsilon)y_{k-1}^2\right] &&\leq  M\E\left[y_{k-1}^2\right], \\
            \label{proof:theorem:ineq_1_4}
            0 & \leq \E\left[ T_2(y_{k-1},\varepsilon)\right] &&\leq  F.
        \end{alignat}
    \end{subequations}
    We now use the aforementioned bounds, as well as \eqref{proof:theorem:ineq_2}, to obtain a form suitable for the use of \Cref{lemma:matrix_upper_lower_bounds} for the dynamics \eqref{lemma:second_moment_terms:2} in \Cref{lemma:second_moment_terms} and the ones in \Cref{lemma:second_moment_aux_terms}.
    Define for ease of notation ${\Bar{C}:= \rho^2\frac{\rho^2\chi + 2\gamma^2}{4}}$ and observe that $${\E\left[y_{k+1}^2\right] = A_{ms}^{(4,*)}\E\left[\zeta_k\right] + R_y(\varepsilon)},$$ as given in \eqref{lemma:second_moment_terms:2}. Then, by using \eqref{proof:theorem:ineq_2} and \eqref{proof:theorem:ineq_1_1} - \eqref{proof:theorem:ineq_1_4}, as well as ${\frac{\mu^2\gamma^2\psi}{4}\varepsilon^2 \geq \frac{\mu^2\gamma^2\psi}{4}\varepsilon\epsilon_0}$ and ${\frac{\mu^2\gamma^2\psi}{4}\varepsilon^2 \leq \frac{\mu^2\gamma^2\psi}{4}\varepsilon\varepsilon_0}$, we have
    \begin{subequations}
    \begin{align}
        \label{proof:theorem:pre_y^2.1}
        R_y(\varepsilon) &\geq \mu^2\varepsilon\Big(\frac{\gamma^2\psi}{4}\varepsilon - \frac{\rho^2\chi M}{2}\left(\E\left[x_{k-1}^2\right] + \E\left[x_{k}^2\right]\right)\Big) \notag \\
        &\geq \frac{\mu^2\gamma^2\psi}{4}\varepsilon\epsilon_0 -\mu^2\varepsilon \frac{\rho^2\chi M}{2}\left(\E\left[x_{k-1}^2\right] + \E\left[x_{k}^2\right]\right) \notag \\
        &=:  \varepsilon Q_1^{(4,*)}\E\left[\zeta_k\right] + \varepsilon b_1^{(4)}  \\
        \label{proof:theorem:pre_y^2.2}
        R_y(\varepsilon) &\leq \mu^2\varepsilon\Bigg(\frac{\gamma^2\psi}{4}\left(\varepsilon + \frac{1}{2} + 2\E\left[y_{k-1}^2\right]\right) \notag + \Bar{C}M\E\left[y_{k-1}^2\right] \\ &+ \frac{\rho^2\gamma^2}{2}F\Bigg) \notag \\ 
        &\leq \mu^2\varepsilon\Bigg(\frac{\gamma^2\psi}{4}\left(\varepsilon_0 + \frac{1}{2}\right) + \left(\frac{\gamma^2\psi}{2} + \Bar{C}M\right)\E\left[y_{k-1}^2\right] \\ &+ \frac{\rho^2\gamma^2}{2}F\Bigg)
        =: \varepsilon Q_2^{(4,*)}\E\left[\zeta_k\right] + \varepsilon b_2^{(4)},
    \end{align}
    \end{subequations}
     with $Q_1^{(4,*)}, b_1^{(4)}$ and $Q_2^{(4,*)}, b_2^{(4)}$ implicitly defined in \eqref{proof:theorem:pre_y^2.1} and \eqref{proof:theorem:pre_y^2.2} respectively, analogous to \eqref{proof:main_theorem:implicit_definitions}. Thereby, we get 
    \begin{subequations}
    \label{proof:theorem_1:inequality_y^2}
        \begin{align}
            \E[y_{k+1}^2] &\geq A_{ms}^{(4,*)}\E\left[\zeta_k\right] + \varepsilon Q_1^{(4,*)}\E\left[\zeta_k\right] + \varepsilon b_1^{(4)}, \\
            \E[y_{k+1}^2] &\leq  A_{ms}^{(4,*)}\E\left[\zeta_k\right] + \varepsilon Q_2^{(4,*)}\E\left[\zeta_k\right] + \varepsilon b_2^{(4)}.
        \end{align}
    \end{subequations}
     Finally, from \Cref{lemma:second_moment_aux_terms} we have $$\E[x_{k+1}y_{k+1}] = A_{ms}^{(6,*)}\E\left[\zeta_k\right] + \frac{\mu\gamma\psi\varepsilon}{2}(\varepsilon + 2\E\left[|y_{k-1}|\right]),$$ and by using \eqref{proof:theorem:ineq_2}, $\E\left[|y_{k-1}|\right] \geq 0$, ${\frac{\mu\gamma\psi}{2}\varepsilon^2 \geq \frac{\mu\gamma\psi}{2}\varepsilon\epsilon_0}$ and ${\frac{\mu\gamma\psi}{2} \varepsilon^2\leq \frac{\mu\gamma\psi}{2}\varepsilon\varepsilon_0}$ we get
    \begin{subequations}
        \begin{align}
            \label{proof:theorem:xy.1}
            \E[x_{k+1}y_{k+1}] &\geq A_{ms}^{(6,*)}\E\left[\zeta_k\right] + \frac{\mu\gamma\psi}{2}\varepsilon^2 \notag \\
            &\geq A_{ms}^{(6,*)}\E\left[\zeta_k\right] + \frac{\mu\gamma\psi}{2}\varepsilon\epsilon_0 \notag \\
            &=: A_{ms}^{(6,*)}\E\left[\zeta_k\right] + \varepsilon Q_1^{(6,*)}\E\left[\zeta_k\right] + \varepsilon b_1^{(6)} \\
            \label{proof:theorem_1:inequality_xy}
            \E[x_{k+1}y_{k+1}] &\leq  A_{ms}^{(6,*)}\E\left[\zeta_k\right] + \frac{\mu\gamma\psi}{4}\varepsilon\big(2\varepsilon +  1 \notag \\& + 4\E\left[y_{k-1}^2\right]\big) \notag \\
            &\leq  A_{ms}^{(6,*)}\E\left[\zeta_k\right] + \frac{\mu\gamma\psi}{4}\varepsilon\big(2\varepsilon_0 +  1 \notag \\ &+ 4\E\left[y_{k-1}^2\right]\big) \notag \\
            &=: A_{ms}^{(6,*)}\E\left[\zeta_k\right] + \varepsilon Q_2^{(6,*)}\E\left[\zeta_k\right] + \varepsilon b_2^{(6)},
        \end{align}
    \end{subequations}
    with $Q_1^{(6,*)}, b_1^{(6)}$ and $Q_2^{(6,*)}, b_2^{(6)}$ implicitly defined in \eqref{proof:theorem:xy.1} and \eqref{proof:theorem_1:inequality_xy} respectively, analogous to \eqref{proof:main_theorem:implicit_definitions}.
    
    \noindent Define $b_1^{(1)} = b_1^{(3)} = b_1^{(5)} = b_2^{(1)} = b_2^{(3)} = b_2^{(5)} := 0$, 
    \begin{align*}
        Q_1^{(1,*)} &:= \begin{bmatrix}
            0 & 0 & 0 & 0 & 0 & 0
        \end{bmatrix} 
    \end{align*}
    with $Q_1^{(1,*)} = Q_1^{(3,*)} = Q_1^{(5,*)} = Q_2^{(1,*)} = Q_2^{(3,*)} =Q_2^{(5,*)}$.
    Trivially, by the Squeeze Theorem for $i \in \{1,3,5\}$ we reformulate $\E\left[\zeta_{k+1}^{(i)}\right] = A_{ms}^{(i,*)}\E\left[\zeta_k\right]$ as $$\E\left[\zeta_{k+1}^{(i)}\right] \leq  \left(A_{ms}^{(i,*)}+ \varepsilon Q_2^{(i,*)}\right)\E\left[\zeta_k\right] + \varepsilon b_2^{(i)} = A_{ms}^{(i,*)}\E\left[\zeta_k\right],$$  $$\E\left[\zeta_{k+1}^{(i)}\right] \geq  \left(A_{ms}^{(i,*)} + \varepsilon Q_1^{(i,*)}\right)\E\left[\zeta_k\right] + \varepsilon b_1^{(i)} = A_{ms}^{(i,*)}\E\left[\zeta_k\right],$$ as $Q_1^{(i,*)}, Q_2^{(i,*)}$ and $b_1^{(i)}, b_2^{(i)}$ contain only zero elements. Finally, we get the element-wise inequalities
    \begin{align*}
       (A_{ms} + \varepsilon Q_1)\E[\zeta_k] + \varepsilon b_1 &\leq \E[\zeta_{k+1}] \\ &\leq (A_{ms} + \varepsilon Q_2)\E[\zeta_k] + \varepsilon b_2.
    \end{align*}
    We first have by assumption $\beta > \mu \rho \sqrt{\chi\psi} \overset{\ref{proposition:general_properties:3}}{=}\mu\gamma\rho$, therefore from \Cref{lemma:expectation_convergence} we know that the system trajectory $(x_k,y_k)$ converges to $(0,0)$ in expectation for any initial condition. Next, from \Cref{lemma:stable_matrices} we know that for any positive real $\chi, \psi$, and $\mu$ there exists some $\rho_0 \in \R^{+}$ and $\beta_0 \in (0,1)$, such that for any $\rho \in (0, \rho_0]$ and $\beta \in [1-\beta_0, 1+\beta_0]$ the eigenvalues of $A_{ms}$ are inside the unit circle. As ${\varepsilon > 0}$ may be chosen arbitrarily small, i.e. from an arbitrarily small interval of positive values,  and by setting~${\theta_k := \E\left[\zeta_k\right]}$ we observe the conditions of \Cref{lemma:matrix_upper_lower_bounds} to be satisfied. Therefore, we have for any initial~$\E\left[\zeta_0\right]$ exponential convergence of the $\E\left[\zeta_k\right]$ to an~$\varepsilon$-dependent set around $0$, which in addition implies that~${\E\left[\tilde{x}_k^2\right]}$ and $\E\left[y_k^2\right]$ converge to an~$\varepsilon$-neighbourhood of~$0$, thus proving the theorem claim for global exponential practical convergence of the variance of the system iterates.
\end{proof}

\end{document}